\documentclass[10pt,reqno]{amsart}
\usepackage{amsmath,amssymb,amsthm,graphicx,epstopdf,mathrsfs,url}
\usepackage[usenames,dvipsnames]{color}
\usepackage{dsfont}   
\definecolor{darkred}{rgb}{0.7,0.1,0.1}
\definecolor{darkblue}{rgb}{0.1,0.1,0.4}
\definecolor{darkgrey}{rgb}{0.5,0.5,0.5}
\usepackage[colorlinks=true,linkcolor=darkred,citecolor=Blue]{hyperref}

\numberwithin{equation}{section}
\theoremstyle{plain}
\newtheorem{thm}{Theorem}[section]
\newtheorem{lem}[thm]{Lemma}

\newtheorem{prop}[thm]{Proposition}
\newtheorem{cor}[thm]{Corollary}

\newtheorem{definition}[thm]{Definition}
\theoremstyle{remark}
\newtheorem{remark}[thm]{Remark}

\newtheorem{example}[thm]{Example}
\theoremstyle{plain}

\newcommand{\hyp}[1]{$C^{2}$-hypersurface as in Definition~\ref{definition_hypersurface}}

\newcommand{\dom}{\mathrm{dom}\,}

\begin{document}
\title[]{Approximation of magnetic Schr\"odinger operators with $\delta$-interactions supported on networks}
\author[]{}


\author[M. Holzmann]{Markus Holzmann}
\address{Institut f\"{u}r Angewandte Mathematik\\
Technische Universit\"{a}t Graz\\
 Steyrergasse 30, 8010 Graz, Austria\\
E-mail: {\tt holzmann@math.tugraz.at}}

\begin{abstract}
  This paper deals with the approximation of a magnetic Schr\"odinger operator with a singular $\delta$-potential that is formally given by $(i \nabla + A)^2 + Q + \alpha \delta_\Sigma$ by Schr\"odinger operators with regular potentials in the norm resolvent sense. This is done for $\Sigma$ being the finite union of $C^2$-hypersurfaces, for coefficients $A$, $Q$, and $\alpha$ under almost minimal assumptions such that the associated quadratic forms are closed and sectorial, and $Q$ and $\alpha$ are allowed to be complex-valued functions. In particular, $\Sigma$ can be a graph in $\mathbb{R}^2$ or the boundary of a piecewise $C^2$-domain. Moreover, spectral implications of the mentioned convergence result are discussed.  
\end{abstract}


\maketitle

\section{Introduction}

Schr\"odinger operators with singular $\delta$-potentials supported on sets of measure zero play an important role in mathematical physics. Singular potentials supported on points were already used in the early days of quantum mechanics \cite{KP31, T35} and later systematically analyzed from a mathematical point of view, see the monograph \cite{AGHH05} and the references therein. In the last three decades, there was a substantial interest in singular potentials supported on hypersurfaces $\Sigma$ in $\mathbb{R}^n$, $n \geq 2$, such as curves in $\mathbb{R}^2$ or surfaces in $\mathbb{R}^3$, see the review paper \cite{E08}, the monograph \cite{EK15}, or the research papers \cite{BEHL19, BEL14, BLL13, BEKS94, E05, EI01, EK03, MPS16}, to mention just a few selected works.
Such operators are formally given by
\begin{equation} \label{def_H_alpha_formal_intro}
  H_{A,Q,\alpha} = (i \nabla + A)^2 + Q + \alpha \delta_\Sigma,
\end{equation}
and they describe for attractive interactions (i.e. for $\alpha< 0$) the propagation of a quantum particle on $\Sigma$ under the influence of a magnetic vector potential $A$ and an electric potential $Q$ allowing the tunneling of the particle between different parts of $\Sigma$, see \cite{BEHL19, E08}; from this point of view, graph like structures $\Sigma \subset \mathbb{R}^2$ are of particular interest. Moreover, operators of the form $H_{0,0,\alpha}$ are useful in the design of photonic crystals with prescribed spectral gaps \cite{FK98, HL18}.

Of course, the operators with the singular potentials are only idealized replacements of more realistic operators with regular potentials. The relation of the regular and the singular models was justified via an approximation procedure already in \cite{KP31}. For $A=Q=0$ the approximation of one, two, and three-dimensional Schr\"odinger operators with $\delta$-point interactions by Schr\"odinger operators with scaled regular potentials in the norm resolvent sense was treated systematically in the 1970s and 1980s, see the monograph \cite{AGHH05} and the references therein; the latter convergence implies, in particular, the convergence of the spectra and of functions of the involved operators \cite{K95, W00} and thus, it justifies the usage of the idealized models as good approximations of their regular counterparts. The approximation of Schr\"odinger operators with $\delta$-potentials supported on the sphere in $\mathbb{R}^3$ was  considered in \cite{AGS87, S92}, then the situation that $\Sigma$ is the boundary of a star shaped domain and $\alpha$ is real-valued and smooth was investigated in \cite{P95}, and unbounded and $C^2$-smooth $\Sigma$ in $\mathbb{R}^2$ and $\mathbb{R}^3$ and constant $\alpha \in \mathbb{R}$ were treated in \cite{EI01, EK03}. It is worth mentioning that in \cite{SV96} Laplacians with $\delta$- (and other measure valued) potentials supported on quite general geometries were approximated, but only in the strong resolvent sense, which does not imply the convergence of the spectra of the approximating operators in the Hausdorff metric. Finally, in \cite{BEHL17} the approximation of the Schr\"odinger operator with a $\delta$-potential of real-valued strength $\alpha \in L^\infty(\Sigma)$ on one bounded or unbounded $C^2$-smooth hypersurface $\Sigma \subset \mathbb{R}^n$, $n \geq 2$, was shown, and in \cite{BEHL19} the approximation of a two-dimensional Schr\"odinger operator with a homogeneous magnetic field and an electric $\delta$-potential of real-valued strength $\alpha \in L^\infty(\Sigma)$ supported on a bounded $C^{1,1}$-smooth curve was treated; note that \cite{BEHL19} is the only work known to the author where the approximation of $H_{A,Q,\alpha}$ for $A \neq 0$ was treated.

The main motivation in this paper is to show that the operator $H_{A,Q,\alpha}$ formally given by~\eqref{def_H_alpha_formal_intro} can be approximated in the norm resolvent sense by Schr\"odinger operators with scaled regular potentials under mild assumptions on $A, Q$, and $\alpha$, and for a wide class of interaction supports $\Sigma$. Concerning the coefficients, from the point of view of quantum mechanics real-valued functions that are associated with self-adjoint operators are favored, but the analysis in this paper works equally well for complex-valued functions $Q$ and $\alpha$, yielding a relation to works on Schr\"odinger operators with non-real $\delta$-potentials \cite{BLLR18, F17, KK17}. Moreover, the interaction support $\Sigma$ will be allowed to be the finite union of bounded or unbounded $C^2$-smooth hypersurfaces, and so classes of graph-like structures in $\mathbb{R}^2$, that are particularly important for the application to leaky quantum graphs, and piecewise $C^2$-smooth $\Sigma \subset \mathbb{R}^n$, which are an important subclass of Lipschitz  domains, are contained.

To set the stage in a rigorous way, let $N, n \in \mathbb{N}$ be fixed with $n \geq 2$ and $\Sigma^{(1)}, \dots, \Sigma^{(N)} \subset \mathbb{R}^n$ be bounded or unbounded hypersurfaces as in Definition~\ref{definition_hypersurface} with normal vector fields $\nu^{(1)}, \dots, \nu^{(N)}$ such that the Hausdorff measure of $\Sigma^{(k)} \cap \Sigma^{(l)}$ is zero for $k \neq l$. Then, the network $\Sigma$ is
\begin{equation} \label{Sigma_intro}
  \Sigma := \bigcup_{k=1}^N \Sigma^{(k)}.
\end{equation}
Moreover, assume that 
\begin{equation} \label{assumption_A}
  A \in L^2_{\textup{loc}}(\mathbb{R}^n; \mathbb{R}^n)
\end{equation}
and $Q = Q_1 + Q_2$ with
\begin{equation} \label{assumption_Q}
  0 \leq Q_1 \in L^1_{\textup{loc}}(\mathbb{R}^n), \, Q_2 \in L^{p'}(\mathbb{R}^n) + L^\infty(\mathbb{R}^n) \text{ with } p' \begin{cases} \geq n/2, &  n>2, \\ > 1, & n=2, \end{cases}
\end{equation}
are fixed. Note that $Q_1$ is automatically real-valued, while $Q_2$ is allowed to be complex-valued. Next, define the space 
\begin{equation} \label{def_H_A_Q_intro}
  \mathcal{H}_{A, Q} := \left\{ f \in L^2(\mathbb{R}^n): |(i \nabla + A) f|, Q_1^{1/2} f \in L^2(\mathbb{R}^n) \right\}.
\end{equation}
This space is studied in detail in Section~\ref{section_Sobolev_space}, where it is shown that $\mathcal{H}_{A, Q}$, when endowed with a proper norm, is a Hilbert space and that the domain of the quadratic form which is associated with the multiplication by $Q_2$ contains $\mathcal{H}_{A, Q}$; cf. Lemma~\ref{lemma_Q_2}. Moreover, in Section~\ref{section_hypersurfaces} it is proved that functions $f \in \mathcal{H}_{A, Q}$ possess Dirichlet traces $f|_{\Sigma^{(k)}} \in L^q(\Sigma^{(k)})$ with $q = 2 \frac{n-1}{n-2+\mu}$ and  $\mu \in (0,1]$.
Eventually, assume that 
\begin{equation} \label{assumption_alpha}
  \alpha := (\alpha^{(k)})_{k=1}^N \in L^p(\Sigma) + L^\infty(\Sigma) := \bigoplus_{k=1}^N \bigl(L^p(\Sigma^{(k)}) + L^\infty(\Sigma^{(k)})\bigr) 
\end{equation}
with $p = \frac{n-1}{1-\gamma}$,
where $\gamma \in (0,1)$ is fixed. To introduce the Schr\"odinger operator with a $\delta$-potential of strength $\alpha$ supported on the network $\Sigma$, consider the quadratic form
\begin{equation} \label{def_H_alpha_form_intro}
  \mathfrak{h}_{A,Q,\alpha}[f] := \int_{\mathbb{R}^n} \bigl(\bigl|(i \nabla + A) f\bigr|^2 + Q |f|^2 \bigr) \textup{d} x + \int_\Sigma \alpha \bigl|f|_\Sigma\bigr|^2 \textup{d} \sigma, \quad \dom \mathfrak{h}_{A,Q,\alpha} = \mathcal{H}_{A, Q},
\end{equation}
where the shortcut notation $\int_\Sigma \alpha |f|_\Sigma|^2 \textup{d} \sigma := \sum_{k=1}^N \int_{\Sigma^{(k)}} \alpha^{(k)} |f|_{\Sigma^{(k)}}|^2 \textup{d} \sigma$ is used. We will see in Proposition~\ref{proposition_delta_form} that $\mathfrak{h}_{A,Q,\alpha}$ is a densely defined, closed, and sectorial form in $L^2(\mathbb{R}^n)$. The associated $m$-sectorial operator $H_{A, Q, \alpha}$ is the rigorous mathematical operator that is associated with the formal expression~\eqref{def_H_alpha_formal_intro}.

In order to construct the sequence of operators $H_{A, Q, V_\varepsilon}$ that converges to $H_{A, Q, \alpha}$ in the norm resolvent sense, note first that there exists $\beta > 0$ such that for each $k \in \{ 1, \dots, N \}$ the mapping
\begin{equation} \label{def_iota_k_intro}
  \iota^{(k)}: \Sigma^{(k)} \times (-\beta, \beta) \rightarrow \mathbb{R}^n, \quad \iota^{(k)}(x_{\Sigma^{(k)}}, t) := x_{\Sigma^{(k)}} + t \nu^{(k)}(x_{\Sigma^{(k)}}),
\end{equation}
is injective, see Proposition~\ref{proposition_tubular_neighbourhood}. Introduce for $\varepsilon \in (0, \beta]$ the set $\Omega_{\varepsilon}^{(k)} := \iota^{(k)}(\Sigma^{(k)} \times (-\varepsilon, \varepsilon))$, and choose a function $V^{(k)} \in L^p(\mathbb{R}^n) + L^\infty(\mathbb{R}^n)$ with $p = \frac{n-1}{1-\gamma}$ for a $\gamma \in (0,1)$ that is supported in $\Omega_{\beta}^{(k)}$. Then, define for $k \in \{ 1, \dots, N \}$
\begin{equation*}
  V_{\varepsilon}^{(k)}(x) := \begin{cases} \frac{\beta}{\varepsilon} V^{(k)}(x_{\Sigma^{(k)}} + \frac{\beta}{\varepsilon} t \nu^{(k)}(x_{\Sigma^{(k)}})), & \text {if } x = x_{\Sigma^{(k)}} + t \nu^{(k)}(x_{\Sigma^{(k)}}) \in \Omega_{\varepsilon}^{(k)}, \\ 0, & \text{if } x \notin \Omega_{\varepsilon}^{(k)}, \end{cases} 
\end{equation*}
and
\begin{equation} \label{def_V_epsilon}
  V_\varepsilon := \sum_{k=1}^N V_{\varepsilon}^{(k)};
\end{equation}
it follows from Proposition~\ref{proposition_tubular_neighbourhood}~(iv) that $V_\varepsilon \in L^p(\mathbb{R}^n) + L^\infty(\mathbb{R}^n)$. Consider the quadratic form
\begin{equation} \label{def_H_V_eps_form_intro}
  \mathfrak{h}_{A,Q,V_\varepsilon}[f] := \int_{\mathbb{R}^n} \bigl(\bigl|(i \nabla + A) f\bigr|^2 + (Q + V_\varepsilon) |f|^2 \bigr) \textup{d} x, \quad \dom \mathfrak{h}_{A,Q,\alpha} = \mathcal{H}_{A, Q}.
\end{equation}
It will turn out as a byproduct in the proof of Theorem~\ref{theorem_approximation} that for sufficiently small $\varepsilon > 0$ the form $\mathfrak{h}_{A,Q,V_\varepsilon}$ is densely defined, closed, and sectorial in $L^2(\mathbb{R}^n)$. The associated $m$-sectorial operator $H_{A, Q, V_\varepsilon}$ is the rigorous mathematical operator that is associated with the expression $(i\nabla + A)^2 + Q + V_\varepsilon$. Now, we are prepared to formulate the main result of this paper:

\begin{thm} \label{theorem_approximation}
  Let $A$ and $Q = Q_1 +Q_2$ be as in~\eqref{assumption_A} and~\eqref{assumption_Q}, respectively, $\Sigma$ be as in~\eqref{Sigma_intro}, and  $V^{(k)} \in L^p(\mathbb{R}^n) + L^\infty(\mathbb{R}^n)$ with $p = \frac{n-1}{1-\gamma}$ for a $\gamma \in (0,1)$ be supported in $\Omega_{\beta}^{(k)}$. Let $V_\varepsilon$ be defined for $\varepsilon \in (0, \beta)$ as in~\eqref{def_V_epsilon} and introduce $\alpha = (\alpha^{(k)})_{k=1}^N \in L^p(\Sigma) + L^\infty(\Sigma)$ by
  \begin{equation} \label{alpha_limit}
    \alpha^{(k)}(x_{\Sigma^{(k)}}) := \int_{-\beta}^\beta V^{(k)}(x_{\Sigma^{(k)}} + t \nu^{(k)}(x_{\Sigma^{(k)}})) \textup{d} t.
  \end{equation}
  Eventually, let $H_{A, Q, \alpha}$ be the $m$-sectorial operator that is associated with the form in~\eqref{def_H_alpha_form_intro}, and let $\gamma_1 \in (0, 2 \gamma] \cap (0,1)$. Then, there exists $\varepsilon_0 \in (0, \beta)$ such that for all $\varepsilon \in (0, \varepsilon_0)$ the form $\mathfrak{h}_{A, Q, V_\varepsilon}$ in~\eqref{def_H_V_eps_form_intro} is closed and sectorial, and for the associated $m$-sectorial operators $H_{A, Q, V_\varepsilon}$ there exist $\lambda_0 < 0$  such that $(-\infty, \lambda_0) \subset \rho(H_{A, Q, \alpha}) \cap \bigcap_{\varepsilon \in (0, \varepsilon_0)} \rho(H_{A, Q, V_{\varepsilon}})$ and for all $\lambda < \lambda_0$ and $\varepsilon \in (0, \varepsilon_0)$ the estimate
  \begin{equation} \label{equation_convergence}
    \big\| (H_{A, Q, \alpha} - \lambda)^{-1} - (H_{A, Q, V_\varepsilon} - \lambda)^{-1} \big\| \leq C \varepsilon^{\gamma_1/2}
  \end{equation}
  with a constant $C = C(\lambda, \varepsilon_0) > 0$ is true. In particular, $H_{A, Q, V_\varepsilon}$ converges in the norm resolvent sense to $H_{A, Q, \alpha}$, as $\varepsilon \searrow 0$.
\end{thm}

The proof of Theorem~\ref{theorem_approximation} follows a strategy that was used in \cite[Appendix~B]{BEHL19} and relies on an estimate of $\mathfrak{h}_{A,Q,V_\varepsilon} - \mathfrak{h}_{A,Q,\alpha}$ in a suitable topology and on applying an abstract result from \cite[Section~VI.3]{K95}. In particular, the fact that only estimates of the associated forms are used allows to treat non-smooth coefficients $A,Q,\alpha$ and networks $\Sigma$ and so,
Theorem~\ref{theorem_approximation} improves the results on the approximation of Schr\"odinger operators with $\delta$-potentials supported on hypersurfaces in several directions:
\begin{itemize}
  \item[(a)] The support of the $\delta$-potential is allowed to be a finite network as in~\eqref{Sigma_intro}. This includes, in particular, also the boundaries of piecewise $C^2$-domains, which are an important class of Lipschitz hypersurfaces, and in dimension $n=2$ also graphs and curves with cusps.
  \item[(b)] The interaction strengths are allowed to be non-real and belong to $L^p$-spaces with almost optimal $p$ so that the operator $H_{A, Q, \alpha}$ can be defined with the help of its associated quadratic form. 
  \item[(c)] Additive electric and magnetic potentials $Q$ and $A$, that are independent of $\varepsilon$, are allowed, and these potentials belong to optimal classes of potentials such that the unperturbed operator $H_{A, Q, 0}$ can be defined as an $m$-sectorial operator with the help of the associated form.
\end{itemize}
It is worth to note that also for the particularly important self-adjoint case Theorem~\ref{theorem_approximation} improves the known approximation results in all points (a)--(c) listed above, as in the previous literature similar results were discussed only for one hypersurface $\Sigma$ (instead of a network) in \cite{BEHL17, EI01, EK03, P95, S92}, bounded interaction strengths $\alpha \in L^\infty(\Sigma)$ in  \cite{BEHL17, EI01, EK03, P95, S92}, the homogeneous magnetic vector potential $A$ in dimension $d=2$ in \cite{BEHL19}, and bounded electric potentials $Q \in L^\infty(\mathbb{R}^d)$ in \cite{BEHL17}.

In Theorem~\ref{theorem_approximation} we start with given potentials $V^{(k)}$, which are used to construct the scaled potentials $V_\varepsilon$ in~\eqref{def_V_epsilon}, and show that the associated operators $H_{A,Q,V_\varepsilon}$ converge in the norm resolvent sense to $H_{A,Q,\alpha}$, where $\alpha$ is computed in terms of $V^{(k)}$ in~\eqref{alpha_limit}. In applications, also the inverse question is of relevance: Given a network $\Sigma$ of hypersurfaces and $\alpha \in L^p(\Sigma) + L^\infty(\Sigma)$, is it possible to approximate $H_{A,Q,\alpha}$ in the norm resolvent sense by operators of the form $H_{A,Q,V_\varepsilon}$? In the following corollary, which follows immediately from Theorem~\ref{theorem_approximation}, a positive answer to this question is given.

\begin{cor} \label{corollary_inverse_result}
  Let $A$ and $Q = Q_1 + Q_2$ be as in~\eqref{assumption_A} and~\eqref{assumption_Q}, respectively, $\Sigma$ be as in~\eqref{Sigma_intro}, and  $\alpha^{(k)} \in L^p(\Sigma^{(k)}) + L^\infty(\Sigma^{(k)})$ with $p = \frac{n-1}{1-\gamma}$ for a $\gamma \in (0,1)$. Let $V^{(k)} \in L^p(\mathbb{R}^n) + L^\infty(\mathbb{R}^n)$ be defined by
  \begin{equation} \label{V_corollary} 
    V^{(k)}(x) = \begin{cases} \frac{1}{2 \beta} \alpha(x_{\Sigma^{(k)}}), & \text{if } x = x_{\Sigma^{(k)}} + t \nu^{(k)}(x_{\Sigma^{(k)}}) \in \Omega_{\beta}^{(k)}, \\ 0, & \text{if } x \notin \Omega_{\beta}^{(k)}, \end{cases}
  \end{equation}
  and let $V_\varepsilon$ be as in~\eqref{def_V_epsilon}.
  Eventually, let $H_{A, Q, \alpha}$ be the $m$-sectorial operator that is associated with the form in~\eqref{def_H_alpha_form_intro}, and let $\gamma_1 \in (0, 2 \gamma] \cap (0,1)$. Then, there exists $\varepsilon_0 \in (0, \beta)$ such that for all $\varepsilon \in (0, \varepsilon_0)$ the form $\mathfrak{h}_{A, Q, V_\varepsilon}$ in~\eqref{def_H_V_eps_form_intro} is closed and sectorial, and for the associated $m$-sectorial operators $H_{A, Q, V_\varepsilon}$ there exist $\lambda_0 < 0$  such that $(-\infty, \lambda_0) \subset \rho(H_{A, Q, \alpha}) \cap \bigcap_{\varepsilon \in (0, \varepsilon_0)} \rho(H_{A, Q, V_{\varepsilon}})$ and for all $\lambda < \lambda_0$ and $\varepsilon \in (0, \varepsilon_0)$ the estimate
  \begin{equation*}
    \big\| (H_{A, Q, \alpha} - \lambda)^{-1} - (H_{A, Q, V_\varepsilon} - \lambda)^{-1} \big\| \leq C \varepsilon^{\gamma_1/2}
  \end{equation*}
  with a constant $C = C(\lambda, \varepsilon_0) > 0$ is true. In particular, $H_{A, Q, V_\varepsilon}$ converges in the norm resolvent sense to $H_{A, Q, \alpha}$, as $\varepsilon \searrow 0$.
\end{cor}

As it is known that convergence of operators in the norm resolvent sense implies the convergence of the spectra of the associated operators \cite{K95, W00}, one can use the results from Theorem~\ref{theorem_approximation} and Corollary~\ref{corollary_inverse_result} to transfer results on the spectrum of $H_{A,Q,\alpha}$ to $H_{A,Q,V_\varepsilon}$. Some examples of this idea are demonstrated in Section~\ref{section_spectral_implications}; here, the focus lies on examples for which the improvements of this paper over the existing literature on the approximation of Schr\"odinger operators with singular potentials are necessary.

The paper is organized as follows. Section~\ref{section_prelim} is of preliminary nature. In Subsection~\ref{section_Sobolev_space} some useful properties of the space $\mathcal{H}_{A,Q}$ from~\eqref{def_H_A_Q_intro} are summarized, Subsection~\ref{section_hypersurfaces} is devoted to the class of $C^2$-hypersurfaces that build the network $\Sigma$, and in Subsection~\ref{section_H_delta} the operator $H_{A,Q,\alpha}$ formally given by~\eqref{def_H_alpha_formal_intro} is introduced. Section~\ref{section_convergence} is devoted to the proof of the main result of this paper, Theorem~\ref{theorem_approximation}.
Finally, in Section~\ref{section_spectral_implications} some spectral implications that follow from the convergence in Theorem~\ref{theorem_approximation} are discussed.

\subsection*{Notations.}

In this paper, $C$ denotes a generic constant that may change its value in between lines. For an open set $\Omega \subset \mathbb{R}^n$ the set $H^1(\Omega)$ is the $L^2$-based Sobolev space of once weakly differentiable functions and $C^k_b(\Omega)$, $k \in \mathbb{N}$, is the set of all $k$ times continuously differentiable functions with bounded partial derivatives up to order $k$. For a closed linear operator $A$ in a Hilbert space, its domain, resolvent set, spectrum, and point spectrum are denoted by $\dom A$, $\rho(A)$, $\sigma(A)$, and $\sigma_{\textup{p}}(A)$. If $A$ is self-adjoint, then its discrete and essential spectrum are $\sigma_{\textup{disc}}(A)$ and $\sigma_{\textup{ess}}(A)$, respectively.

\section{Preliminaries} \label{section_prelim}

\subsection{The space $\mathcal{H}_{A, Q}$} \label{section_Sobolev_space}

Throughout this subsection, let $A \in L^2_{\textup{loc}}(\mathbb{R}^n; \mathbb{R}^n)$
and $Q = Q_1 + Q_2$ with $0 \leq Q_1 \in L^1_{\textup{loc}}(\mathbb{R}^n)$, $Q_2 \in L^{p'}(\mathbb{R}^n) + L^\infty(\mathbb{R}^n)$, where $p' \geq \frac{n}{2}$ for $n>2$ and $p' > 1$ for $n=2$,
be as in~\eqref{assumption_A} and~\eqref{assumption_Q}, respectively. Note that $A$ and $Q_1$ are real-valued, while $Q_2$ is allowed to be complex-valued. In this subsection, some useful properties of the space
\begin{equation} \label{def_H_A_Q}
  \mathcal{H}_{A, Q} := \left\{ f \in L^2(\mathbb{R}^n): |(i \nabla + A) f|, Q_1^{1/2} f \in L^2(\mathbb{R}^n) \right\},
\end{equation}
where $(i \nabla + A) f$ is understood in the sense of distributions, are summarized. Define the norm
\begin{equation} \label{def_norm_H_A_Q}
  \| f \|_{\mathcal{H}_{A,Q}}^2 := \int_{\mathbb{R}^n} \bigl(|(i \nabla + A) f|^2 + (Q_1+1) |f|^2 \bigr) \textup{d} x.
\end{equation}
Basic properties of $\mathcal{H}_{A, Q}$ are stated in the following lemma, see \cite[Lemma~1 and Theorem~1]{LS81} for a proof.

\begin{lem} \label{lemma_H_A_Q_basic}
  The space $\mathcal{H}_{A, Q}$ endowed with the norm in~\eqref{def_norm_H_A_Q} is a Hilbert space and $C_0^\infty(\mathbb{R}^n)$ is dense in $\mathcal{H}_{A,Q}$.
\end{lem}

Next, we state a variant of the diamagnetic inequality that will be useful in our considerations. 

\begin{lem} \label{lemma_diamagnetic_inequality}
  For all $f \in \mathcal{H}_{A,Q}$ one has $\nabla |f| \in L^2(\mathbb{R}^n; \mathbb{C}^n)$ and the estimate 
  \begin{equation} \label{diamagnetic_inequality}
    \big\| \nabla |f| \big\|_{L^2(\mathbb{R}^n)} \leq \| f \|_{\mathcal{H}_{A,Q}}
  \end{equation}
  is true.
\end{lem}
\begin{proof}
  By \cite[Theorem~7.21]{LL01} for almost all $x \in \mathbb{R}^n$ the pointwise estimate 
  \begin{equation*}
    \big| \nabla |f(x)| \big| \leq \big| (i \nabla + A(x)) f(x) \big|
  \end{equation*}
  is true. Therefore, for $f \in \mathcal{H}_{A,Q}$ one has $\nabla |f| \in L^2(\mathbb{R}^n; \mathbb{C}^n)$ and as $Q_1 \geq 0$ one concludes
  \begin{equation*}
    \begin{split}
      \int_{\mathbb{R}^n} \big| \nabla |f| \big|^2 \textup{d} x &\leq \int_{\mathbb{R}^n} \big| (i \nabla + A) f \big|^2 \textup{d} x \\
      &\leq \int_{\mathbb{R}^n} \bigl(|(i \nabla + A) f|^2 + (Q_1+1) |f|^2 \bigr) \textup{d} x = \| f \|_{\mathcal{H}_{A,Q}}^2,
    \end{split}
  \end{equation*}
  which is the claimed result.
\end{proof}

In the following lemma the non-real term $Q_2$ is added and it is shown, that the norm that is induced by the associated quadratic form is equivalent to the norm in~\eqref{def_norm_H_A_Q}.

\begin{lem} \label{lemma_Q_2}
   Let $Q_2 \in L^{p'}(\mathbb{R}^n)+L^\infty(\mathbb{R}^n)$, where $p' \geq \frac{n}{2}$ for $n>2$ and $p' > 1$ for $n=2$. Then, the form 
   \begin{equation} \label{unperturbed_form}
     \mathfrak{h}_{A,Q}[f] := \int_{\mathbb{R}^n} \bigl(|(i \nabla + A) f|^2 + (Q_1+Q_2) |f|^2 \bigr) \textup{d} x, \quad \dom \mathfrak{h}_{A,Q} = \mathcal{H}_{A,Q},
   \end{equation}
   is well-defined, sectorial, and closed, and there exist constants $C_2 > C_1 > 0$ and $C_3 \in \mathbb{R}$ such that
   \begin{equation*}
     C_1 \| f \|_{\mathcal{H}_{A,Q}}^2 \leq \bigl(\textup{Re}\, \mathfrak{h}_{A,Q} + C_3 \bigr)[f] \leq C_2 \| f \|_{\mathcal{H}_{A,Q}}^2, \quad f \in \mathcal{H}_{A,Q}.
   \end{equation*}
\end{lem}
\begin{proof}
  Consider the quadratic form 
  \begin{equation} \label{def_q_2_form}
    \mathfrak{q}_2[f] := \int_{\mathbb{R}^n} Q_2 |f|^2 \textup{d} x, \quad \dom \mathfrak{q}_2 = \bigl\{ f \in L^2(\mathbb{R}^n): |Q_2|^{1/2} f \in L^2(\mathbb{R}^n) \bigr\}.
  \end{equation}
  To show the claims, it suffices to prove that $\mathfrak{q}_2$ is form bounded with respect to the form $\mathcal{H}_{A,Q} \ni f \mapsto \| f \|_{\mathcal{H}_{A,Q}}^2$ with form bound smaller than one. For this, we claim that the multiplication by $|Q_2|^{1/2}$ gives rise to a bounded operator from  $H^1(\mathbb{R}^n)$ to $L^2(\mathbb{R}^n)$ and that for all $\delta > 0$ there exists $C_\delta > 0$ such that
  \begin{equation} \label{multiplication_op_H1}
    \big\| |Q_2|^{1/2} f \big\|_{L^2(\mathbb{R}^n)}^2 \leq \delta \| f \|_{H^1(\mathbb{R}^n)}^2 + C_\delta \| f \|_{L^2(\mathbb{R}^n)}^2, \quad f \in H^1(\mathbb{R}^n).
  \end{equation}
  To see this, note that by assumption there exist $Q_{2,p'} \in L^{p'}(\mathbb{R}^n)$ and $Q_{2,\infty} \in L^\infty(\mathbb{R}^n)$ such that $Q_2 = Q_{2,p'} + Q_{2,\infty}$. 
  It follows from \cite[Lemma~3.2]{ABHS25} (applied for $s=1$) that there exists for all $\delta > 0$ a constant $\widetilde{C}_\delta$ such that for all $ f \in H^1(\mathbb{R}^n)$
  \begin{equation*} 
    \big\| |Q_{2,p'}|^{1/2} f \big\|_{L^2(\mathbb{R}^n)}^2 \leq \delta \| f \|_{H^1(\mathbb{R}^n)}^2 + \widetilde{C}_\delta \| f \|_{L^2(\mathbb{R}^n)}^2.
  \end{equation*}
  Therefore,
  \begin{equation*} 
    \begin{split}
    \big\| |Q_{2}|^{1/2} f \big\|_{L^2(\mathbb{R}^n)}^2 
    &= \int_{\mathbb{R}^n} |Q_2||f|^2 \textup{d} x \\
    &\leq \int_{\mathbb{R}^n} |Q_{2,p'}||f|^2 \textup{d} x + \int_{\mathbb{R}^n} |Q_{2,\infty}||f|^2 \textup{d} x \\    
    &\leq \delta \| f \|_{H^1(\mathbb{R}^n)}^2 + \bigl(\widetilde{C}_\delta + \| Q_{2,\infty} \|_{L^\infty(\mathbb{R}^n)} \bigr) \| f \|_{L^2(\mathbb{R}^n)}^2,
    \end{split}
  \end{equation*}
  which is exactly~\eqref{multiplication_op_H1} with $C_\delta := \widetilde{C}_\delta + \|Q_{2,\infty} \|_{L^\infty(\mathbb{R}^n)}$.

  Using~\eqref{multiplication_op_H1} for a fixed $\delta \in (0,1)$ and Lemma~\ref{lemma_diamagnetic_inequality} imply first for $f \in C_0^\infty(\mathbb{R}^n)$
  \begin{equation*}
    \begin{split}
      \bigl| \mathfrak{q}_2[f] \bigr| &\leq \big\| |Q_2|^{1/2} |f| \big\|_{L^2(\mathbb{R}^n)}^2 \\
      &\leq \delta \big\| \nabla |f| \big\|_{L^2(\mathbb{R}^n)}^2 + (C_\delta + \delta + 1) \big\| |f|  \big\|_{L^2(\mathbb{R}^n)}^2 \\
      &\leq \delta \| f \|_{\mathcal{H}_{A, Q}}^2 + (C_\delta + \delta + 1) \big\| f  \big\|_{L^2(\mathbb{R}^n)}^2.
    \end{split}
  \end{equation*}
  Since $C_0^\infty(\mathbb{R}^n)$ is dense in $\mathcal{H}_{A, Q}$ by Lemma~\ref{lemma_H_A_Q_basic}, the latter estimate remains true for all $f \in \mathcal{H}_{A, Q}$. This finishes the proof.
\end{proof}

\begin{remark} \label{remark_form_bound}
  The only property of the term $Q_2$ that is needed in this paper and, in particular, to prove Theorem~\ref{theorem_approximation}, is the one from Lemma~\ref{lemma_Q_2}. As seen in the proof, the statement remains true, if $Q_2$ is such that the form $\mathfrak{q}_2$ in~\eqref{def_q_2_form} is form bounded with respect to the form $\mathcal{H}_{A,Q} \ni f \mapsto \| f \|_{\mathcal{H}_{A,Q}}^2$ with form bound smaller than one. Hence, Theorem~\ref{theorem_approximation} remains true for all $Q_2$ which have the latter property.
\end{remark}

\subsection{Description of hypersurfaces and their tubular neighborhoods} \label{section_hypersurfaces}

In this section the hypersurfaces that build the network being the support of the $\delta$-potential are described, and some properties of the map $\iota$ in~\eqref{def_iota_k_intro} which describes the tubular neighborhood of the network are collected. Finally, a suitable variant of the trace theorem for these hypersurfaces and the space $\mathcal{H}_{A, Q}$ from~\eqref{def_H_A_Q} is discussed. Since for the proof of Theorem~\ref{theorem_approximation} similar ideas as in \cite[Section~3.2]{BHSL25} are used, we follow closely the definition of hypersurfaces from \cite{BHSL25}:

\begin{definition} \label{definition_hypersurface}
  We say that $\Sigma \subset \mathbb{R}^n$ is an admissible $C^2$-hypersurface, if there exist a number $p \in \mathbb{N}$, open sets $\Omega, V_1, \dots, V_p, W_1, \dots, W_p \subset \mathbb{R}^n$, real-valued mappings $\zeta_1, \dots, \zeta_p \in C^2_b(\mathbb{R}^{n-1})$, rotation matrices $\kappa_1, \dots, \kappa_p \in \mathbb{R}^{n \times n}$, and $\varepsilon_1 > 0$ such that the following holds:
  \begin{itemize}
    \item[(i)] $\partial \Omega \subset \bigcup_{k=1}^p W_k$.
    \item[(ii)] For all $x \in \partial \Omega$ there exists $k=k(x) \in \{ 1, \dots, p \}$ such that $B(x,\varepsilon_1) \subset W_k$.
    \item[(iii)] $W_k \cap \Omega = W_k \cap \Omega_k$ for all $k \in \{ 1, \dots, p \}$, where the notation $\Omega_k = \{ \kappa_k (x', x_n): x' \in \mathbb{R}^{n-1}, x_n < \zeta_k(x') \}$ is used.
    \item[(iv)] $\Sigma \subset \partial \Omega \cap \big( \bigcup_{k=1}^N V_k \big)$.
    \item[(v)] $V_k \subset W_k$ and $\partial \Omega \cap V_k = \Sigma \cap V_k$ for all $k \in \{ 1, \dots, p \}$.
  \end{itemize}
\end{definition}

Note that the sets $\Omega_1, \dots, \Omega_p$ appearing in item~(iii) of the previous definition are rotated Lipschitz hypographs as in \cite[Chapter~3]{M00}, where the function $\zeta_k$ belongs to $C^2_b(\mathbb{R}^{n-1})$. 
Roughly speaking, conditions (i)--(iii) in the above definition mean that the set $\Omega$ is of the same form as the sets $\Omega_+$ considered in \cite{BHSL25}, while conditions (iv)\&(v) mean that the admissible hypersurface $\Sigma$ is a subset of $\partial \Omega$ with good properties, so that an integral over $\Sigma$ is well-defined. In particular, (relatively open subsets of) $C^2_b$-hypographs and the boundaries of compact $C^2$-domains are admissible $C^2$-hypersurfaces. It will be important for our purposes that item~(ii) in the previous definition implies in the case of unbounded hypersurfaces that different ends can not be too close to each other. A typical example of an unbounded curve in $\mathbb{R}^2$ which violates this condition is stated in the following example:

\begin{example}
  Let $\Omega \subset \mathbb{R}^2$ be an unbounded domain with boundary $\Sigma = \partial \Omega$ which can be described as in items~(i) \& (iii) of Definition~\ref{definition_hypersurface} and which has the two ends $\pm (\sin^2(x) + x^{-1})$ for $x$ sufficiently large. Note that the two ends must be the boundaries of two different sets $\Omega_k$ as in Definition~\ref{definition_hypersurface}~(iii), as they can not be the image of one function $\mathbb{R} \ni t \mapsto \kappa (t, \zeta(t))$ with a rotation matrix $\kappa$ and $\zeta \in C^2_b(\mathbb{R})$; without loss of generality, the two ends are parts of the boundaries of 
  \begin{equation*}
    \Omega_1 := \big\{ (x, x_2): x \in \mathbb{R}, x_2 < \sin^2(x) + x^{-1} \big\}
  \end{equation*}
  and 
  \begin{equation*}
    \Omega_2 := \big\{ ( -x, -x_2 ): x \in \mathbb{R}, x_2 < \sin^2(x) - x^{-1} \big\}.
  \end{equation*}
  Set $x_n = (n \pi, (n \pi)^{-1})$. Note that on the one hand item~(iii) in Definition~\ref{definition_hypersurface} implies for the associated sets $W_1, W_2$ that $W_k \cap \Omega_k \subset \Omega$, $k \in \{ 1,2 \}$.  On the other hand, $x_n \in \partial \Omega_1$ for $n$ sufficiently large and for a fixed $\varepsilon_1 > 0$ one has $B(x_n, \varepsilon_1) \cap \Omega_1 \not\subset \Omega$, implying $B(x_n, \varepsilon_1) \not\subset W_1$; in the same way, one argues that $B(x_n, \varepsilon_1) \not\subset W_2$. Therefore, this curve $\Sigma$ does not satisfy assertion~(ii) in Definition~\ref{definition_hypersurface}.
\end{example}

Sobolev spaces $H^s(\Sigma)$ of order $s \in [-2, 2]$ on admissible a $C^2$-hypersurface $\Sigma$ are constructed in the standard way, with the help of a suitable partition of unity subordinate to the open cover $V_1, \dots, V_p$ of $\Sigma$ and local coordinates, see \cite{BHSL25, M00} for details. 

An important role will be played by the Weingarten map or shape operator $W$ on an admissible $C^2$-hypersurface $\Sigma$. This is the linear map at the tangent space $T_{x_\Sigma}$ for $x_{\Sigma} = \kappa_k (x', \zeta_k(x')) \in \Sigma$ via its action on a basis element $\partial_{x'_j} \kappa_k (x', \zeta_k(x'))$ of $T_{x_\Sigma}$,
\begin{equation} \label{def_Weingarten_map}
  W(x_{\Sigma}) \partial_{x'_j} \kappa_k (x', \zeta_k(x')) = - \partial_{x'_j} \nu\bigl( \kappa_k (x', \zeta_k(x')) \bigr), \quad j \in \{ 1, \dots, n-1\},
\end{equation}
where $\nu$ denotes the normal vector field at $\Sigma$. It is not difficult to show that the definition of $W(x_{\Sigma})$ is independent of the parametrization of $\Sigma$. Some useful properties of $W$ and the map 
\begin{equation} \label{def_iota}
  \iota: \Sigma \times \mathbb{R} \rightarrow \mathbb{R}^n, \quad \iota(x_\Sigma, t) := x_\Sigma + t \nu(x_\Sigma),
\end{equation}
see also~\eqref{def_iota_k_intro}, are summarized in the following proposition.

\begin{prop}\label{proposition_tubular_neighbourhood} 
  Let $\Sigma \subset \mathbb{R}^n$ be an admissible $C^2$-hypersurface with unit normal vector field $\nu$, $W$ be the associated Weingarten map, and $\iota$ be defined by~\eqref{def_iota}. Then, there exists $\beta > 0$ such that the following holds:
  \begin{itemize}
    \item[$\textup{(i)}$] $\iota |_{\Sigma \times (-\beta, \beta)}$ is injective.
    \item[$\textup{(ii)}$] There exists $C>0$ such that $|1-\det(I-\varepsilon W(x_{\Sigma}))| \leq C \varepsilon <1/2$ for all $x_{\Sigma} \in \Sigma$ and $\varepsilon \in (-\beta, \beta)$.
    \item[$\textup{(iii)}$] For $\varepsilon \in (-\beta, \beta)$ one has with $\Omega_\varepsilon := \iota (\Sigma \times (-\varepsilon, \varepsilon))$ that $f \in L^1(\Omega_\varepsilon)$ if and only if  $f \circ \iota |_{\Sigma \times (-\varepsilon,\varepsilon)} \in L^1(\Sigma\times(-\varepsilon,\varepsilon))$  and in this case
    \begin{equation*}
      \int_{\Omega_{\varepsilon}} f(x)  \textup{d} x = \int_{-\varepsilon}^{\varepsilon} \int_\Sigma f(\iota(x_\Sigma, t))  \det(I- tW(x_\Sigma))\textup{d}\sigma(x_{\Sigma}) \textup{d} t.
    \end{equation*}
    \item[$\textup{(iv)}$] Let  $p \geq 1$ and $\varepsilon \in (-\beta, \beta)$. Then, $f \in L^p(\Omega_\varepsilon)$ if and only if $f \circ \iota|_{\Sigma \times (-\varepsilon,\varepsilon)} \in L^p(\Sigma\times(-\varepsilon,\varepsilon))$ and there exist constants $0 < C_1 < C_2$ such that
    \begin{equation*}
      C_1 \bigl\| f \circ \iota|_{\Sigma \times (-\varepsilon,\varepsilon)} \bigr\|_{L^p(\Sigma\times(-\varepsilon,\varepsilon))} \leq \| f \|_{L^p(\Omega_\varepsilon)} \leq C_2 \bigl\| f \circ \iota|_{\Sigma \times (-\varepsilon,\varepsilon)} \bigr\|_{L^p(\Sigma\times(-\varepsilon,\varepsilon))}.
    \end{equation*}
  \end{itemize}
\end{prop}
\begin{proof}
  Items (i)--(iii) can be shown in the same way as in  \cite[Proposition~2.4]{BHSL25}. The claim in (iv) is an immediate consequence of the formula in (iii) and (ii).
\end{proof}

In the next proposition a variant of the trace theorem involving the space $\mathcal{H}_{A,Q}$ defined by~\eqref{def_H_A_Q}, which is a key ingredient to define the for $\mathfrak{h}_{A,Q, \alpha}$ in~\eqref{def_H_alpha_form_intro}, is stated. 

\begin{prop} \label{proposition_trace_theorem}
  Let $\Sigma \subset \mathbb{R}^n$ be an admissible $C^2$-hypersurface and $q = 2 \frac{n-1}{n-2+\mu}$ for $\mu \in (0,1]$. Then, the map $C_0^\infty(\mathbb{R}^n) \ni f \mapsto f|_\Sigma$ admits a unique continuous extension $\gamma_D: \mathcal{H}_{A,Q} \rightarrow L^q(\Sigma)$ and for all $\delta > 0$ there exists $C_\delta > 1$ such that
  \begin{equation} \label{trace_estimate}
    \| \gamma_D f \|_{L^q(\Sigma)}^2 \leq \delta \| f \|_{\mathcal{H}_{A,Q}}^2 + C_\delta \| f \|_{L^2(\mathbb{R}^n)}^2
  \end{equation}
  holds for all $f \in \mathcal{H}_{A,Q}$.
\end{prop}
\begin{proof}
  The proof follows closely ideas from \cite[Lemma~A.1]{BLS23}. Let $f \in C_0^\infty(\mathbb{R}^n)$. It follows from \cite[Theorem 8.12.6.I]{B12} that $L^q(\Sigma)$ is continuously embedded in the Sobolev space $H^{1/2-\mu/2}(\Sigma)$, which is again continuously embedded in $H^{1/2-\mu'/2}(\Sigma)$ for $\mu' \in (0, \mu] \cap (0,1)$. Therefore, the trace theorem for Sobolev spaces \cite[Theorem~2]{M87} implies that 
  \begin{equation} \label{trace_theorem_Sobolev}
    \begin{split}
      \| f|_\Sigma \|_{L^q(\Sigma)}^2 &= \bigl\| |f|\big|_\Sigma \bigr\|_{L^q(\Sigma)}^2 \\
      &\leq C \bigl\| |f|\big|_\Sigma \bigr\|_{H^{1/2-\mu/2}(\Sigma)}^2 \\
      & \leq C \bigl\| |f|\big|_\Sigma \bigr\|_{H^{1/2-\mu'/2}(\Sigma)}^2 \\
      & \leq C \bigl\| |f| \bigr\|_{H^{1-\mu'/2}(\mathbb{R}^n)}^2 \\
      &\leq \delta \bigl\| |f| \bigr\|_{H^1(\mathbb{R}^n)}^2 + (C_\delta-1) \bigl\| |f| \bigr\|_{L^2(\mathbb{R}^n)}^2 \\
      &\leq \delta \| f \|_{\mathcal{H}_{A,Q}}^2 + C_\delta \| f \|_{L^2(\mathbb{R}^n)}^2,
    \end{split}
  \end{equation}
  where \cite[Theorem~3.30]{HT08} or 
\cite[Satz~11.18~e)]{W00} was used in the second last step, and Lemma~\ref{lemma_diamagnetic_inequality} in the last step. Since $C_0^\infty(\mathbb{R}^n)$ is dense in $\mathcal{H}_{A,Q}$ by Lemma~\ref{lemma_H_A_Q_basic}, the claim of this proposition follows from~\eqref{trace_theorem_Sobolev}.
\end{proof}

\subsection{Rigorous definition of $H_{A,Q,\alpha}$} \label{section_H_delta}

Throughout this subsection, assume that $A \in L^2_{\textup{loc}}(\mathbb{R}^n; \mathbb{R}^n)$
and $Q = Q_1 + Q_2$ with $0 \leq Q_1 \in L^1_{\textup{loc}}(\mathbb{R}^n)$ and $Q_2 \in L^{p'}(\mathbb{R}^n) + L^\infty(\mathbb{R}^n)$, where $p' \geq \frac{n}{2}$ for $n>2$ and $p' > 1$ for $n=2$,
are as in~\eqref{assumption_A} and~\eqref{assumption_Q}, respectively. Moreover, let $N \in \mathbb{N}$, $\Sigma^{(1)}, \dots, \Sigma^{(N)}$ be admissible $C^2$-hypersurfaces as in Definition~\ref{definition_hypersurface} such that the Hausdorff measure of $\Sigma^{(k)} \cap \Sigma^{(l)}$ is zero for $k \neq l$, and $\Sigma := \bigcup_{k=1}^N \Sigma^{(k)}$. 
For $p \in [1, \infty]$ the $L^p$-spaces on $\Sigma$ are defined by
\begin{equation*}
  L^p(\Sigma) := \bigoplus_{k=1}^N L^p(\Sigma^{(k)})
\end{equation*}
and for $f = (f_k)_{k=1}^N \in L^1(\Sigma)$ its integral is
\begin{equation*}
  \int_\Sigma f \text{d} \sigma := \sum_{k=1}^N \int_{\Sigma^{(k)}} f_k \text{d} \sigma.
\end{equation*}

The goal in this subsection is to study the form $\mathfrak{h}_{A,Q,\alpha}$ from~\eqref{def_H_alpha_form_intro}, which then leads to a rigorous definition of the formal expression in~\eqref{def_H_alpha_formal_intro}. The proof of the following proposition follows similar ideas as in the construction of non-magnetic Schr\"odinger operators with $\delta$-potentials, see, for instance, \cite[Proposition~2.1]{BLS23}, \cite[Section~4]{BEKS94} or \cite[Lemma~2.7]{BEHL17}.

\begin{prop} \label{proposition_delta_form}
  Let $\gamma \in (0,1)$ be fixed, $p = \frac{n-1}{1-\gamma}$, and
  \begin{equation*} 
    \alpha := (\alpha^{(k)})_{k=1}^N \in L^p(\Sigma) + L^\infty(\Sigma) .
  \end{equation*}
  Then, the quadratic form
  \begin{equation*} 
    \mathfrak{h}_{A,Q,\alpha}[f] := \int_{\mathbb{R}^n} \bigl(\bigl|(i \nabla + A) f\bigr|^2 + Q |f|^2 \bigr) \textup{d} x + \int_\Sigma \alpha |\gamma_D f|^2 \textup{d} \sigma, \quad \dom \mathfrak{h}_{A,Q,\alpha} = \mathcal{H}_{A, Q},
  \end{equation*}
  is densely defined, sectorial, and closed. Moreover, there exist constants $C_1, C_2 > 0$ such that
  \begin{equation} \label{norm_estimate_delta_form}
    \| f \|_{\mathcal{H}_{A,Q}}^2 \leq C_1 \big(\textup{Re}\, \mathfrak{h}_{A,Q,\alpha}[f] + C_2 \| f \|_{L^2(\mathbb{R}^n)}^2\big)
  \end{equation}
  holds for all $f \in \mathcal{H}_{A,Q}$.
\end{prop}

The $m$-sectorial operator that is associated with $\mathfrak{h}_{A,Q,\alpha}$ via the first representation theorem \cite[Theorem~VI.2.1]{K95} is denoted by $H_{A,Q,\alpha}$, i.e.
\begin{equation} \label{def_H_alpha}
  \begin{split}
    H_{A,Q,\alpha} f &= \bigl((i \nabla + A)^2 + Q\bigr) f \text{ in } \mathbb{R}^n \setminus \Sigma, \\
    \dom  H_{A,Q,\alpha} &= \bigl\{ f \in \mathcal{H}_{A,Q}: \bigl((i \nabla + A)^2 + Q\bigr) f \in L^2(\mathbb{R}^n \setminus \Sigma), \\
    &\qquad\quad  \mathfrak{h}_{A,Q,\alpha}[f, g] = \bigl( \bigl((i \nabla + A)^2 + Q\bigr) f, g \bigr)_{L^2(\mathbb{R}^n)} \, \forall g \in \mathcal{H}_{A, Q} \bigr\};
  \end{split}
\end{equation}
here, the fact that $H_{A,Q,\alpha}$ acts as $(i \nabla + A)^2 + Q$ in $\mathbb{R}^n \setminus \Sigma$ follows from the inclusion $C_0^\infty(\mathbb{R}^n \setminus \Sigma) \subset \mathcal{H}_{A,Q}$, see Lemma~\ref{lemma_H_A_Q_basic}, and the definition of the distributional expression $(i \nabla + A)^2 + Q$; cf. \cite[Lemma~5]{LS81} for a similar argument.

\begin{proof}[Proof of Proposition~\ref{proposition_delta_form}]
  By assumption, for each $k \in \{ 1, \dots, N \}$ there exist functions $\alpha^{(k)}_p \in L^p(\Sigma^{(k)})$ and $\alpha^{(k)}_\infty \in L^\infty(\Sigma^{(k)})$ such that $\alpha^{(k)} = \alpha^{(k)}_p + \alpha^{(k)}_\infty$. Set $\alpha_p := (\alpha^{(k)}_p)_{k=1}^N \in L^p(\Sigma)$ and $\alpha_\infty := (\alpha^{(k)}_\infty)_{k=1}^N \in L^\infty(\Sigma)$, so that $\alpha= \alpha_p + \alpha_\infty$. Define for $* \in \{ p, \infty \}$ the form
  \begin{equation*}
    \mathfrak{a}_*[f] := \sum_{k=1}^N \int_{\Sigma^{(k)}} \alpha^{(k)}_* |\gamma_D f|^2 \textup{d} \sigma, \quad \dom \mathfrak{a}_* = \mathcal{H}_{A, Q}.
  \end{equation*}
  Let $\delta > 0$ and $f \in \mathcal{H}_{A, Q}$ be fixed. Using H\"older's inequality with $p = \frac{n-1}{1-\gamma}$ and $q = (1-\frac{1}{p})^{-1} =  \frac{n-1}{n-2+\gamma}$ and Proposition~\ref{proposition_trace_theorem} (with $\mu = \gamma$) one finds that 
  \begin{equation*}
    \begin{split}
      \big|\mathfrak{a}_p[f]\big| &\leq \left( \sum_{k=1}^N \int_{\Sigma^{(k)}} |\alpha^{(k)}_p|^p \textup{d} \sigma \right)^{1/p} \left( \sum_{k=1}^N \int_{\Sigma^{(k)}} |\gamma_D f|^{2q} \textup{d} \sigma \right)^{1/q} \\
      &\leq  \| \alpha_p\|_{L^p(\Sigma)} \bigl(\delta \| f \|_{\mathcal{H}_{A, Q}}^2 + C_\delta \| f \|_{L^2(\mathbb{R}^n)}^2 \bigr).
    \end{split}
  \end{equation*}
  Likewise, Proposition~\ref{proposition_trace_theorem} applied for $\mu = 1$ yields
  \begin{equation*}
    \begin{split}
      \big|\mathfrak{a}_\infty[f]\big| \leq  \| \alpha_\infty\|_{L^\infty(\Sigma)} \bigl(\delta \| f \|_{\mathcal{H}_{A, Q}}^2 + C_\delta \| f \|_{L^2(\mathbb{R}^n)}^2 \bigr).
    \end{split}
  \end{equation*}
  In particular, the forms $\mathfrak{a}_p$ and $\mathfrak{a}_\infty$ are both well-defined and form bounded with bound smaller than $\frac{1}{2}$ with respect to $\mathfrak{h}_{A, Q}$ defined by~\eqref{unperturbed_form}. Therefore, by the KLMN theorem \cite[Theorem~VI.1.33]{K95} the form $\mathfrak{h}_{A, Q, \alpha} = \mathfrak{h}_{A, Q} + \mathfrak{a}_p + \mathfrak{a}_\infty$ is sectorial and closed, as $\mathfrak{h}_{A,Q}$ has these properties by Lemma~\ref{lemma_Q_2}, and the estimate~\eqref{norm_estimate_delta_form} is true.
\end{proof}

\begin{remark} \label{remark_h_V_eps_closed}
  By using, e.g., \cite[Lemma~3.2]{ABHS25} instead of Proposition~\ref{proposition_trace_theorem} one can show in the same way as in the proof of Proposition~\ref{proposition_delta_form} that the form $\mathfrak{h}_{A,Q,V_\varepsilon}$ in~\eqref{def_H_V_eps_form_intro} is closed and sectorial. However, this will follow for sufficiently small $\varepsilon > 0$ automatically as a byproduct in the proof of Theorem~\ref{theorem_approximation} and thus, we do not provide the proof of this result here.
\end{remark}

\section{Proof of Theorem~\ref{theorem_approximation}} \label{section_convergence}

The following preparatory lemma, in which the traces of functions in $\mathcal{H}_{A, Q}$ on $\Sigma$ and a slightly shifted hypersurface in the normal direction are compared, is the key ingredient to prove the main result of this paper, Theorem~\ref{theorem_approximation}. In the proof, we use properties of the shift in the normal direction - a strategy that is inspired by \cite[Section~3.2]{BHSL25}. Recall that for an admissible $C^2$-hypersurface $\Sigma$ as in Definition~\ref{definition_hypersurface} the map $\iota$ is defined by~\eqref{def_iota}.

\begin{lem} \label{lemma_shift}
  Let $A \in L^2_{\textup{loc}}(\mathbb{R}^n; \mathbb{R}^n)$ and $Q = Q_1 + Q_2$ with $0 \leq Q_1 \in L^1_{\textup{loc}}(\mathbb{R}^n)$ and $Q_2 \in L^{p'}(\mathbb{R}^n)$, where $p' \geq \frac{n}{2}$ for $n>2$ and $p' > 1$ for $n=2$. Moreover, let $\Sigma$ be an admissible $C^2$-hypersurface with unit normal vector field $\nu$, $q =  2 \frac{n-1}{n-2+\mu}$ for $\mu \in (0,1]$, and $\mu' \in (0, \mu] \cap (0,1)$. Then, there exists $\beta_0 > 0$ and a constant $C_1 > 0$ such that for all $r \in (-\beta_0, \beta_0)$ and all $f \in \mathcal{H}_{A,Q}$ the estimate
  \begin{equation} \label{equation_trace_difference}
    \left( \int_{\Sigma} \bigl||f(\iota(x_\Sigma, r))| - |f(\iota(x_\Sigma,0))| \bigr|^q \textup{d} \sigma(x_{\Sigma}) \right)^{1/q} \leq C_1 |r|^{\mu'/2} \| f \|_{\mathcal{H}_{A, Q}}
  \end{equation}
  is true. In particular, there exists a constant $C_2 > 0$ such that for all $r \in (-\beta_0, \beta_0)$ and all $f \in \mathcal{H}_{A,Q}$
  \begin{equation} \label{equation_trace_uniformly_bounded}
    \left( \int_{\Sigma} |f(\iota(x_\Sigma, r))|^{q} \textup{d} \sigma(x_{\Sigma}) \right)^{1/q} \leq C_2 \| f \|_{\mathcal{H}_{A, Q}}.
  \end{equation}
\end{lem}
\begin{proof}
  First, note that it suffices to prove~\eqref{equation_trace_difference}, as this estimate, the triangle inequality, and Proposition~\ref{proposition_trace_theorem}  imply
  \begin{equation*}
    \begin{split}
       \left( \int_{\Sigma} |f(\iota(x_\Sigma, r))|^{q} \textup{d} \sigma(x_{\Sigma}) \right)^{1/q}
       &\leq   \left( \int_{\Sigma} \bigl||f(\iota(x_\Sigma, r))| - |f(\iota(x_\Sigma,0))| \bigr|^q \textup{d} \sigma(x_{\Sigma}) \right)^{1/q} \\
       &\qquad \qquad \qquad \qquad + \left( \int_{\Sigma} |f(\iota(x_\Sigma, 0))|^{q} \textup{d} \sigma(x_{\Sigma}) \right)^{1/q} \\
       &\leq (C_1 |r|^{\mu'/2} + C_3) \| f \|_{\mathcal{H}_{A, Q}},
    \end{split}
  \end{equation*}
  where $C_3 = \delta + C_\delta$ is the constant from Proposition~\ref{proposition_trace_theorem} for a fixed choice of $\delta$. This yields the claim in~\eqref{equation_trace_uniformly_bounded}. The estimate in~\eqref{equation_trace_difference} is shown in two steps. In \textit{Step~1} a suitable Sobolev norm of an extension of $f(\iota(x_\Sigma, r)) - f(\iota(x_\Sigma, 0))$ is estimated, which is used in \textit{Step~2} to deduce~\eqref{equation_trace_difference}.
  
  \textit{Step~1.} Let $\widetilde{\nu}$ be a $C^1_b$-extension of $\nu$ to $\mathbb{R}^n$. To show that such an extension exists, let $\Omega$ be as in Definition~\ref{definition_hypersurface} and $\widehat{\nu}$ be the normal vector field at $\partial \Omega$; then, as $\Sigma \subset \partial \Omega$, the function $\widehat{\nu}$ is an extension of $\nu$ to $\partial \Omega$. Moreover, choose a real-valued function $\phi \in C^1_b(\mathbb{R})$ with $\phi(0) = 1$ and $\textup{supp}\, \phi \subset (-\beta^\Omega, \beta^\Omega)$, where $\beta^\Omega$ is the number $\beta$ from Proposition~\ref{proposition_tubular_neighbourhood} for the admissible hypersurface $\partial \Omega$. Then one possible choice of $\widetilde{\nu}$ is given by
  \begin{equation*}
    \widetilde{\nu}(x) := \begin{cases} \widehat{\nu}(x_{\partial \Omega}) \phi(t), & \text{if } x = x_{\partial \Omega} + t \widehat{\nu}(x_{\partial \Omega}) \text{ for } x_{\partial \Omega} \in \partial \Omega, t \in (-\beta^\Omega, \beta^\Omega), \\0, &\text{otherwise}. \end{cases}
  \end{equation*}
  Define for $f \in H^1(\mathbb{R}^n)$ the function $f_r \in H^1(\mathbb{R}^n)$ by $f_r(x) := f(x + r \widetilde{\nu}(x))$.
  The aim of this step is to prove for all $\delta \in [0,1]$ the estimate
  \begin{equation} \label{equation_shift_operator}
    \| f - f_r \|_{H^{1-\delta}(\mathbb{R}^n)} \leq C |r|^{\delta} \| f \|_{H^1(\mathbb{R}^n)}
  \end{equation}
  for a suitable constant $C > 0$, for which we closely follow the proof of \cite[Proposition~3.3]{BHSL25}.
  
  First, \eqref{equation_shift_operator} is shown for $\delta = 1$. With the main theorem of calculus, one finds for $f \in C_0^\infty(\mathbb{R}^n)$ that
  \begin{equation*}
    \begin{split}
      \| f_r - f \|_{L^2(\mathbb{R}^n)}^2 &= \int_{\mathbb{R}^n} \bigl|f(x + r \widetilde{\nu}(x)) - f(x) \bigr|^2 \textup{d} x \\
      &= \int_{\mathbb{R}^n} \left|\int_0^r \frac{\textup{d}}{\textup{d} s} f(x + s \widetilde{\nu}(x)) \textup{d} s \right|^2 \textup{d} x \\
      &= \int_{\mathbb{R}^n} \left|\int_0^r \nabla f(x + s \widetilde{\nu}(x)) \cdot \widetilde{\nu}(x) \textup{d} s \right|^2 \textup{d} x \\
      &\leq \int_{\mathbb{R}^n} \int_0^{|r|} |\nabla f(x + s \widetilde{\nu}(x))|^2 \textup{d} s \int_0^{|r|} |\widetilde{\nu}(x)|^2 \textup{d} s  \textup{d} x \\
      &\leq |r| \| \widetilde{\nu} \|_{L^\infty(\mathbb{R}^n)}^2  \int_0^{|r|} \int_{\mathbb{R}^n} |\nabla f(x + s \widetilde{\nu}(x))|^2 \textup{d} x \textup{d} s,
    \end{split}
  \end{equation*}
  where the Cauchy-Schwarz inequality and Fubini's theorem were used. To proceed, denote the Jacobi matrix of the mapping $x \mapsto x + s \widetilde{\nu}(x)$ by $D (x + s \widetilde{\nu}(x))$. Then, the eigenvalues of $D (x + s \widetilde{\nu}(x))$ can be estimated from below by $1-|r| \| \widetilde{\nu} \|_{C^1_b(\mathbb{R}^n)}$. Suppose that $\beta_0 \in (0, \beta^\Omega)$ is chosen sufficiently small such that there exists a constant $C>0$ such that $1-|r| \| \widetilde{\nu} \|_{C^1_b(\mathbb{R}^n)} \geq C$ for all $r \in (-\beta_0, \beta_0)$. Then, 
  \begin{equation*}
    \det D (x + s \widetilde{\nu}(x)) \geq (1-|r| \| \widetilde{\nu} \|_{C^1_b(\mathbb{R}^n)})^n \geq C^n
  \end{equation*}
  and we can continue estimating  
  \begin{equation*}
    \begin{split}
      \| f - f_r \|_{L^2(\mathbb{R}^n)}^2 
      &\leq |r| \| \widetilde{\nu} \|_{L^\infty(\mathbb{R}^n)}^2  \int_0^{|r|} \int_{\mathbb{R}^n} |\nabla f(x + s \widetilde{\nu}(x))|^2 \frac{|\det D (x + s \widetilde{\nu}(x))|}{|\det D (x + s \widetilde{\nu}(x))|} \textup{d} x \textup{d} s \\
      &\leq \frac{|r| \| \widetilde{\nu} \|_{L^\infty(\mathbb{R}^n)}^2}{C^{n}}  \int_0^{|r|} \int_{\mathbb{R}^n} |\nabla f(x + s \widetilde{\nu}(x))|^2 |\det D (x + s \widetilde{\nu}(x))| \textup{d} x \textup{d} s \\
      &= \frac{|r| \| \widetilde{\nu} \|_{L^\infty(\mathbb{R}^n)}^2}{C^{n}}  \int_0^{|r|} \int_{\mathbb{R}^n} |\nabla f(y)|^2 \textup{d} y \textup{d} s,
    \end{split}
  \end{equation*}
  where the substitution $y = x + s \widetilde{\nu}(x)$ was used in the last step. Since the last integral is independent of $s$ and $C_0^\infty(\mathbb{R}^n)$ is dense in $H^1(\mathbb{R}^n)$, the claim in~\eqref{equation_shift_operator} for $\delta = 1$ follows.
  
  Next, we prove the statement in~\eqref{equation_shift_operator} for $\delta = 0$. With a similar calculation as in the case $\delta = 1$ and the substitution $y = x + r \widetilde{\nu}(x)$ one finds that
  \begin{equation*}
    \begin{split}
      \| \nabla f_r \|_{L^2(\mathbb{R}^n)}^2
      &= \int_{\mathbb{R}^n} \bigl| (I+r D \widetilde{\nu}(x)) \cdot \nabla f(x + r \widetilde{\nu}(x)) \bigr|^2 \textup{d} x \\
      &\leq (1+|r|\| \widetilde{\nu} \|_{C^1_b(\mathbb{R}^n)})^2  \int_{\mathbb{R}^n} \bigl| \nabla f(x + r \widetilde{\nu}(x)) \bigr|^2 \frac{|\det D (x + r \widetilde{\nu}(x))|}{|\det D (x + r \widetilde{\nu}(x))|} \textup{d} x \\
      &\leq \frac{(1+|r|\| \widetilde{\nu} \|_{C^1_b(\mathbb{R}^n)})^2}{(1-|r| \| \widetilde{\nu} \|_{C^1_b(\mathbb{R}^n)})^n}   \int_{\mathbb{R}^n} |\nabla f(y)|^2 \textup{d} y.
    \end{split}
  \end{equation*}
  By combining this with the triangle inequality and~\eqref{equation_shift_operator} for $\delta=1$ one finds that
  \begin{equation*} 
    \| f - f_r \|_{H^{1}(\mathbb{R}^n)} \leq \| f - f_r \|_{L^2(\mathbb{R}^n)} + \| \nabla f \|_{L^2(\mathbb{R}^n)} + \| \nabla f_r \|_{L^2(\mathbb{R}^n)} \leq C \| f \|_{H^1(\mathbb{R}^n)},
  \end{equation*}
  which is the claim for $\delta = 0$.
  
  Finally, having the statement in~\eqref{equation_shift_operator} for $\delta = 0$ and for $\delta=1$, the claim for $\delta \in (0,1)$ follows with an interpolation argument.
  
  \textit{Step~2.} Let $f \in C_0^\infty(\mathbb{R}^n)$. It follows from \cite[Theorem 8.12.6.I]{B12} that $L^q(\Sigma)$ is continuously embedded in $H^{1/2-\mu/2}(\Sigma)$, which is again continuously embedded in $H^{1/2-\mu'/2}(\Sigma)$. Therefore, the trace theorem for Sobolev spaces \cite[Theorem~2]{M87} implies
  \begin{equation*} 
    \begin{split}
      \left( \int_{\Sigma} \bigl||f(\iota(x_\Sigma, r))| - |f(\iota(x_\Sigma, 0))| \bigr|^q \textup{d} \sigma(x_{\Sigma}) \right)^{1/q} 
      & \leq C \bigl\| \gamma_D(|f|_r - |f|) \big| \bigr\|_{H^{1/2-\mu/2}(\Sigma)} \\
       & \leq C \bigl\| \gamma_D(|f|_r - |f|) \big| \bigr\|_{H^{1/2-\mu'/2}(\Sigma)} \\
      &\leq C \bigl\| |f|_r - |f| \bigr\|_{H^{1-\mu'/2}(\mathbb{R}^n)} \\
      &\leq C |r|^{\mu'/2} \big\| |f| \big\|_{H^1(\mathbb{R}^n)},
    \end{split}
  \end{equation*}
  where~\eqref{equation_shift_operator} was used in the last step. By combining this with Lemma~\ref{lemma_diamagnetic_inequality} one finds that~\eqref{equation_trace_difference} is true for $f \in C_0^\infty(\mathbb{R}^n)$. Since $C_0^\infty(\mathbb{R}^n)$ is dense in $\mathcal{H}_{A, Q}$ by Lemma~\ref{lemma_H_A_Q_basic}, the claim follows for all $f \in \mathcal{H}_{A, Q}$.
\end{proof}

\begin{remark}
  As in the proof of \cite[Theorem~4.5]{BEHL19}, see equation~(B.9) there, one can show that for $q=2$ the estimate in~\eqref{equation_trace_difference} is true for $\mu' = 1$. This can be used to show that the difference in~\eqref{equation_convergence} can be estimated by $C \varepsilon^{1/2}$, if $V_\varepsilon \in L^\infty(\mathbb{R}^n)$. Since a main focus in this paper is on more singular potentials $V_\varepsilon$, this small possible improvement is not further investigated here.
\end{remark}

\begin{proof}[Proof of Theorem~\ref{theorem_approximation}]
  First, note that $V^{(k)} \in L^p(\mathbb{R}^n) + L^\infty(\mathbb{R}^n)$ together with Proposition~\ref{proposition_tubular_neighbourhood}~(iv) implies that $\alpha$ defined by~\eqref{alpha_limit} belongs to $L^p(\Sigma) + L^\infty(\Sigma)$. Consider the quadratic form
  \begin{equation} \label{form_difference}
    \begin{split}
      \mathfrak{a}[f] := \int_{\mathbb{R}^n} V_\varepsilon |f|^2 \textup{d} x - \int_\Sigma \alpha |\gamma_D f|^2 \textup{d} \sigma, \quad \dom \mathfrak{a} = \mathcal{H}_{A, Q}.
    \end{split}
  \end{equation}
  We will show below that $\mathfrak{a}$ is well defined and there exists a constant $C > 0$ such that  for all $\varepsilon \in (0, \varepsilon_0)$ with $\varepsilon_0 \in (0, \beta_0)$, where $\beta_0$ is as in Lemma~\ref{lemma_shift},
  \begin{equation} \label{estimate_difference}
    |\mathfrak{a}[f]| \leq C \varepsilon^{\gamma_1/2} \bigl( \textup{Re}\, \mathfrak{h}_{A, Q, \alpha}[f] +  \| f \|_{L^2(\mathbb{R}^n)}^2 \bigr), \quad f \in \mathcal{H}_{A, Q}.
  \end{equation}
  This implies then first with the KLMN theorem \cite[Theorem~VI.1.33]{K95} for  $\varepsilon \in (0, \varepsilon_0)$ with $\varepsilon_0$ sufficiently small that $\mathfrak{h}_{A, Q, V_\varepsilon} = \mathfrak{h}_{A, Q, \alpha} + \mathfrak{a}$ in~\eqref{def_H_V_eps_form_intro} is closed and sectorial, as $\mathfrak{h}_{A, Q, \alpha}$ has this property by Proposition~\ref{proposition_delta_form}. And second, provided that $\varepsilon_0$ is small enough, the estimate in~\eqref{estimate_difference} and \cite[Theorem~VI.3.4]{K95} show that there exists $\lambda_0 < 0$ such that $(-\infty, \lambda_0) \subset \rho(H_{A, Q, \alpha}) \cap \bigcap_{\varepsilon \in (0, \varepsilon_0)} \rho(H_{A, Q, V_{\varepsilon}})$ and that~\eqref{equation_convergence} is true.
  
  So it remains to show that~\eqref{estimate_difference} is true. It suffices to prove this for $f \in C_0^\infty(\mathbb{R}^n)$, as this set is dense in $\mathcal{H}_{A, Q}$ by Lemma~\ref{lemma_H_A_Q_basic}. Let $f \in C_0^\infty(\mathbb{R}^n)$ be fixed. Recall that each of the functions $V^{(k)}_\varepsilon$ that appear in the definition of $V_\varepsilon$ in~\eqref{def_V_epsilon} is supported in $\Omega_\varepsilon^{(k)} = \{ x_{\Sigma^{(k)}} + t \nu^{(k)}(x_{\Sigma^{(k)}}): x_{\Sigma^{(k)}} \in \Sigma^{(k)}, t \in (-\varepsilon, \varepsilon) \}$ and denote the map defined in~\eqref{def_iota} for $\Sigma^{(k)}$ by $\iota^{(k)}$ (see also~\eqref{def_iota_k_intro}).  
  Using Proposition~\ref{proposition_tubular_neighbourhood}~(iii) at each of the admissible $C^2$-hypersurfaces $\Sigma^{(k)}$ (with associated Weingarten map $W^{(k)}$) for the first integral and~\eqref{alpha_limit} in the second integral, one gets that
    \begin{equation*} 
    \begin{split}
      \mathfrak{a}[f] &= \sum_{k=1}^N \left[ \int_{\Omega_\varepsilon^{(k)}} V_{\varepsilon}^{(k)}(x) |f(x)|^2 \text{d} x - \int_{\Sigma^{(k)}} \alpha(x_{\Sigma^{(k)}}) |f(x_{\Sigma^{(k)}})|^2 \text{d} \sigma(x_{\Sigma^{(k)}}) \right] \\
      &= \sum_{j=1}^N \bigg[ \int_{\Sigma^{(k)}} \int_{-\varepsilon}^\varepsilon \frac{\beta}{\varepsilon} V^{(k)}\bigg( \iota^{(k)} \bigg( x_{\Sigma^{(k)}}, \frac{\beta}{\varepsilon} t \bigg) \bigg) |f(\iota^{(k)}(x_{\Sigma^{(k)}}, t))|^2 \\
      &\qquad \qquad \qquad \qquad \qquad \qquad\qquad \cdot\det(1 - t W^{(k)}({x_{\Sigma^{(k)}}})) \text{d} t \text{d} \sigma(x_{\Sigma^{(k)}}) \\
      &\qquad \qquad - \int_{\Sigma^{(k)}} \int_{-\beta}^\beta V^{(k)}(\iota^{(k)}(x_{\Sigma^{(k)}}, r)) \text{d} r |f(\iota^{(k)}(x_{\Sigma^{(k)}},0))|^2 \text{d} \sigma(x_{\Sigma^{(k)}}) \bigg] \\
      &= \sum_{j=1}^N \int_{\Sigma^{(k)}} \int_{-\beta}^\beta V^{(k)}\big( \iota^{(k)}( x_{\Sigma^{(k)}}, r) \big) \\
      &\qquad \qquad \qquad \cdot\bigg[ \bigg|f\bigg(\iota^{(k)}\bigg(x_{\Sigma^{(k)}}, \frac{\varepsilon}{\beta} r \bigg)\bigg)\bigg|^2 \det\bigg(1 - \frac{\varepsilon}{\beta} r W^{(k)}({x_{\Sigma^{(k)}}})\bigg) \\
      &\qquad \qquad \qquad \qquad\qquad\qquad\qquad- \big|f(\iota^{(k)}(x_{\Sigma^{(k)}},0))\big|^2 \bigg] \text{d} t \text{d} \sigma(x_{\Sigma^{(k)}}), \\
    \end{split}
  \end{equation*}
  where the substitution $r = \frac{\beta}{\varepsilon} t$ was done in the last step. Since by assumption $V^{(k)} \in L^p(\mathbb{R}^n) + L^\infty(\mathbb{R}^n)$, there exist $V^{(k)}_p \in L^p(\mathbb{R}^n)$ and $V_\infty^{(k)} \in L^\infty(\mathbb{R}^n)$ such that $V^{(k)} = V_p^{(k)} + V_\infty^{(k)}$. As $\textup{supp} \, V^{(k)} \subset \Omega_\beta^{(k)}$, it is no restriction to assume for $* \in \{ p, \infty\}$ that also $\textup{supp} \, V_*^{(k)} \subset \Omega_\beta^{(k)}$. Define for $* \in \{ p, \infty \}$
  \begin{equation*}
    \begin{split}
      \mathcal{I}_{*, 1}^{(k)} := \int_{\Sigma^{(k)}} \int_{-\beta}^\beta V^{(k)}_*\big( \iota^{(k)}(x_{\Sigma^{(k)}}, r) \big) \bigg|f\bigg(\iota^{(k)} \bigg( x_{\Sigma^{(k)}}, \frac{\varepsilon}{\beta} r \bigg)\bigg) \bigg|^2 &\\
      \cdot\bigg(\det\bigg(1 - \frac{\varepsilon}{\beta} r W^{(k)}({x_{\Sigma^{(k)}}})-1 \bigg) \bigg) &\text{d} t \text{d} \sigma(x_{\Sigma^{(k)}})
    \end{split}
  \end{equation*}
  and 
  \begin{equation*}
    \begin{split}
      \mathcal{I}_{*, 2}^{(k)} := \int_{\Sigma^{(k)}} \int_{-\beta}^\beta &V^{(k)}_*\big( \iota^{(k)} (x_{\Sigma^{(k)}}, r) \big) \\
      &\quad\cdot \bigg[ \bigg|f\bigg( \iota^{(k)} \bigg(x_{\Sigma^{(k)}}, \frac{\varepsilon}{\beta} r \bigg)\bigg)\bigg|^2 - \big|f(\iota^{(k)}(x_{\Sigma^{(k)}}, 0)) \big|^2 \bigg] \text{d} t \text{d} \sigma(x_{\Sigma^{(k)}}).
    \end{split}
  \end{equation*}
  Then, one can further rewrite
  \begin{equation} \label{decomposition_a}
    \mathfrak{a}[f] = \sum_{k=1}^N \bigl( \mathcal{I}_{p, 1}^{(k)} + \mathcal{I}_{\infty, 1}^{(k)} + \mathcal{I}_{p, 2}^{(k)} + \mathcal{I}_{\infty, 2}^{(k)} \bigr).
  \end{equation}
  With Proposition~\ref{proposition_tubular_neighbourhood}~(ii),~\eqref{equation_trace_uniformly_bounded} (applied for $\mu=1$), and~\eqref{norm_estimate_delta_form} one gets
  \begin{equation} \label{estimate_I_infty_1}
    \begin{split}
      \big| \mathcal{I}_{\infty, 1}^{(k)} \big| &\leq C \varepsilon \| V^{(k)}_\infty \|_{L^\infty(\mathbb{R}^n)} \int_{-\beta}^\beta  \int_{\Sigma^{(k)}} \bigg|f\bigg(\iota^{(k)} \bigg(x_{\Sigma^{(k)}}, \frac{\varepsilon}{\beta} r \bigg)\bigg) \bigg|^2 \text{d} \sigma(x_{\Sigma^{(k)}}) \text{d} t \\
      &\leq C \varepsilon \| f \|_{\mathcal{H}_{A, Q}}^2 \\
      &\leq \varepsilon C \bigl( \textup{Re}\, \mathfrak{h}_{A, Q, \alpha}[f] + \| f \|_{L^2(\mathbb{R}^n)}^2 \bigr).
    \end{split}
  \end{equation}
  Likewise, H\"older's inequality (with $p = \frac{n-1}{1-\gamma}$ and $q = (1-\frac{1}{p})^{-1} = \frac{n-1}{n-2+\gamma}$), Proposition~\ref{proposition_tubular_neighbourhood}~(iv),~\eqref{equation_trace_uniformly_bounded} for $\mu = \gamma$, and~\eqref{norm_estimate_delta_form} imply
  \begin{equation} \label{estimate_I_p_1}
    \begin{split}
      \big| \mathcal{I}_{p, 1}^{(k)} \big| &\leq C \varepsilon \bigg(  \int_{\Sigma^{(k)}} \int_{-\beta}^\beta \big| V^{(k)}_p\big( \iota^{(k)}(x_{\Sigma^{(k)}}, r) \big) \big|^p \text{d} t \text{d} \sigma(x_{\Sigma^{(k)}}) \bigg)^{1/p} \\
      &\quad \qquad  \cdot \bigg( \int_{-\beta}^\beta  \int_{\Sigma^{(k)}} \bigg|f\bigg(\iota^{(k)} \bigg( x_{\Sigma^{(k)}}, \frac{\varepsilon}{\beta} r \bigg)\bigg) \bigg|^{2q} \text{d} \sigma(x_{\Sigma^{(k)}}) \text{d} t \bigg)^{1/q} \\
      &\leq C \varepsilon \| V_p^{(k)} \|_{L^p(\mathbb{R}^n)} \| f \|_{\mathcal{H}_{A, Q}}^2 \\
      &\leq \varepsilon C \bigl( \textup{Re}\, \mathfrak{h}_{A, Q, \alpha}[f] +  \| f \|_{L^2(\mathbb{R}^n)}^2 \bigr).
    \end{split}
  \end{equation}
  Next,  the Cauchy-Schwarz inequality, Fubini's theorem, \eqref{equation_trace_difference}, \eqref{equation_trace_uniformly_bounded} (both applied for $\mu=1$) and~\eqref{norm_estimate_delta_form} yield for $\gamma' \in (0,1)$
  \begin{equation} \label{estimate_I_infty_2}
    \begin{split}
      \big| \mathcal{I}_{\infty, 2}^{(k)} \big| &= \bigg| \int_{\Sigma^{(k)}} \int_{-\beta}^\beta V^{(k)}_\infty\big( \iota^{(k)} (x_{\Sigma^{(k)}}, r) \big) 
      \bigg[ \bigg|f\bigg( \iota^{(k)} \bigg(x_{\Sigma^{(k)}}, \frac{\varepsilon}{\beta} r \bigg)\bigg)\bigg| - \big|f(\iota^{(k)}(x_{\Sigma^{(k)}}, 0)) \big| \bigg] \\
      &\qquad\qquad \qquad \qquad \cdot \bigg[ \bigg|f\bigg( \iota^{(k)} \bigg(x_{\Sigma^{(k)}}, \frac{\varepsilon}{\beta} r \bigg)\bigg)\bigg| + \big|f(\iota^{(k)}(x_{\Sigma^{(k)}}, 0)) \big| \bigg]\text{d} t \text{d} \sigma(x_{\Sigma^{(k)}}) \bigg| \\
      &\leq \| V^{(k)}_\infty \|_{L^\infty(\mathbb{R}^n)} \\
      &\quad \cdot \bigg(\int_{-\beta}^\beta \int_{\Sigma^{(k)}} \bigg| \bigg|f\bigg(\iota^{(k)} \bigg(x_{\Sigma^{(k)}}, \frac{\varepsilon}{\beta} r \bigg)\bigg)\bigg| - \big|f(\iota^{(k)}(x_{\Sigma^{(k)}},0))\big| \bigg|^2  \text{d} \sigma(x_{\Sigma^{(k)}}) \text{d} t \bigg)^{1/2} \\
      &\quad \cdot \bigg(  \int_{-\beta}^\beta \int_{\Sigma^{(k)}} \bigg| \bigg|f\bigg(\iota^{(k)} \bigg(x_{\Sigma^{(k)}}, \frac{\varepsilon}{\beta} r \bigg)\bigg)\bigg| + \big|f(\iota^{(k)}(x_{\Sigma^{(k)}},0))\big| \bigg|^2  \text{d} \sigma(x_{\Sigma^{(k)}}) \text{d} t \bigg)^{1/2} \\
      &\leq C \varepsilon^{\gamma'/2} \| f \|_{\mathcal{H}_{A, Q}}^2 \\
      &\leq C \varepsilon^{\gamma'/2}  \bigl( \textup{Re}\, \mathfrak{h}_{A, Q, \alpha}[f] + \| f \|_{L^2(\mathbb{R}^n)}^2 \bigr).
    \end{split}
  \end{equation}
  Finally, using H\"older's inequality for three functions with $p = \frac{n-1}{1-\gamma}$, $q_1 = 2 \frac{n-1}{n-2+\gamma_1}$ and $q_2 = 2 \frac{n-1}{n-2+\gamma_2}$, where $\gamma_1 \in (0, 2 \gamma] \cap (0,1)$ and $\gamma_2 = 2 \gamma - \gamma_1$, so that $\frac{1}{p} + \frac{1}{q_1} + \frac{1}{q_2} = 1$, Fubini's theorem, Proposition~\ref{proposition_tubular_neighbourhood}~(ii)\&(iii),~\eqref{equation_trace_difference},~\eqref{equation_trace_uniformly_bounded}, and~\eqref{norm_estimate_delta_form} imply
  \begin{equation} \label{estimate_I_p_2}
    \begin{split}
      \big| \mathcal{I}_{p, 2}^{(k)} \big| &\leq \bigg(  \int_{\Sigma^{(k)}} \int_{-\beta}^\beta \big| V^{(k)}_p\big( \iota^{(k)}(x_{\Sigma^{(k)}}, r) \big) \big|^p \text{d} t \text{d} \sigma(x_{\Sigma^{(k)}}) \bigg)^{1/p} \\
      &\quad   \cdot \bigg( \int_{-\beta}^\beta  \int_{\Sigma^{(k)}} \bigg| \bigg|f\bigg(\iota^{(k)} \bigg( x_{\Sigma^{(k)}}, \frac{\varepsilon}{\beta} r \bigg)\bigg) - \big| f(\iota^{(k)}(x_\Sigma^{(k)}, 0)) \big| \bigg|^{q_1} \text{d} \sigma(x_{\Sigma^{(k)}}) \text{d} t \bigg)^{1/q_1} \\
      &\quad   \cdot \bigg( \int_{-\beta}^\beta  \int_{\Sigma^{(k)}} \bigg| \bigg|f\bigg(\iota^{(k)} \bigg( x_{\Sigma^{(k)}}, \frac{\varepsilon}{\beta} r \bigg)\bigg) + \big| f(\iota^{(k)}(x_\Sigma^{(k)}, 0)) \big| \bigg|^{q_2} \text{d} \sigma(x_{\Sigma^{(k)}}) \text{d} t \bigg)^{1/q_2} \\
      &\leq C \varepsilon^{\gamma_1/2} \| V_p^{(k)} \|_{L^p(\mathbb{R}^n)} \| f \|_{\mathcal{H}_{A, Q}}^2 \\
      &\leq C \varepsilon^{\gamma_1/2} \bigl( \textup{Re}\, \mathfrak{h}_{A, Q, \alpha}[f] + \| f \|_{L^2(\mathbb{R}^n)}^2 \bigr).
    \end{split}
  \end{equation}
  By combining the estimates~\eqref{estimate_I_infty_1}--\eqref{estimate_I_p_2} in~\eqref{decomposition_a}, one finds that~\eqref{estimate_difference} is true.
\end{proof}

\section{Consequences for the spectrum of $H_{A,Q,V_\varepsilon}$} \label{section_spectral_implications}

Since convergence in the norm resolvent sense implies the convergence of the spectra of the associated operators, Theorem~\ref{theorem_approximation} allows to transfer results on the spectrum of $H_{A,Q,\alpha}$ to the spectrum of $H_{A,Q,V_\varepsilon}$. If $\Sigma$ is one $C^2$-hypersurface, some interesting consequences are elaborated in \cite[Section~3.3]{BEHL17}; in particular, in \cite[Proposition~3.5]{BEHL17} results on the existence of eigenvalues of $H_{0,0,\alpha}$ (for $A=Q=0$) in various situations are transferred to $H_{0,0,V_\varepsilon}$. In this section, some further consequences are discussed, which can not be obtained from the results in \cite{BEHL17} due to the different assumptions there. Here, the focus lies in self-adjoint examples; however, since convergence in the norm resolvent sense also implies the convergence of the spectrum in the non self-adjoint setting, see for instance \cite[Section IV.3.5]{K95}, also results on the spectrum of $H_{A,Q,\alpha}$ for non-real $\alpha$ from, e.g., \cite{BLLR18, F17, KK17} could be transferred.

The starting point in all of the examples below is the fact that, as a consequence of Theorem~\ref{theorem_approximation}, the discrete and the essential spectra of $H_{A,Q,V_\varepsilon}$ converge to the respective spectra of $H_{A,Q,\alpha}$. The claim in the following lemma follows from known results on the spectral convergence for families of operators that converge in the norm resolvent sense, see \cite[Section IV.3.5]{K95} or \cite[Satz~9.24]{W00} or also the proof of \cite[Proposition~3.5]{BEHL17} for a more detailed argument.

\begin{lem} \label{lemma_convergence_spectrum}
  Let $A$ and $Q = Q_1 +Q_2$ be as in~\eqref{assumption_A} and~\eqref{assumption_Q}, respectively, $\Sigma$ be as in~\eqref{Sigma_intro}, and  $V^{(k)} \in L^p(\mathbb{R}^n) + L^\infty(\mathbb{R}^n)$ with $p = \frac{n-1}{1-\gamma}$ and $\gamma \in (0,1)$ be supported in $\Omega_{\beta}^{(k)}$. Assume additionally that $Q$ and $V^{(k)}$ are real-valued. Let for $\varepsilon \in (0, \varepsilon_0)$ with $\varepsilon_0$ as in Theorem~\ref{theorem_approximation} the potential $V_\varepsilon$ be defined by~\eqref{def_V_epsilon} and $H_{A,Q,V_\varepsilon}$ be the self-adjoint operator associated with the form $\mathfrak{h}_{A,Q,V_{\varepsilon}}$ in~\eqref{def_H_V_eps_form_intro}. Eventually, let $\alpha$ be as in~\eqref{alpha_limit} and $H_{A,Q,\alpha}$ be the self-adjoint operator associated with the form $\mathfrak{h}_{A,Q,\alpha}$ in~\eqref{def_H_alpha_form_intro}. Then, the following is true:
  \begin{itemize}
    \item[\textup{(i)}] $\lambda \in \sigma_{\textup{disc}}(H_{A,Q,\alpha})$ if and only of there exists a family of numbers $(\lambda_\varepsilon)_{\varepsilon \in (0, \varepsilon_0)}$ with $\lambda_\varepsilon \in \sigma_{\textup{disc}}(H_{A,Q,V_\varepsilon})$ such that $\lambda_\varepsilon \rightarrow \lambda$ for $\varepsilon \searrow 0$.
    \item[\textup{(ii)}] $\lambda \in \sigma_{\textup{ess}}(H_{A,Q,\alpha})$ if and only of there exists a family of numbers $(\lambda_\varepsilon)_{\varepsilon \in (0, \varepsilon_0)}$ with $\lambda_\varepsilon \in \sigma_{\textup{ess}}(H_{A,Q,V_\varepsilon})$ such that $\lambda_\varepsilon \rightarrow \lambda$ for $\varepsilon \searrow 0$.
  \end{itemize}
\end{lem}

In the following examples the potentials $V_\varepsilon$ are constructed via piecewise constant functions $V^{(k)}$. So let us fix some notations that are the same in Propositions~\ref{proposition_magnetic}--\ref{proposition_cusp}. Let $\Sigma = \bigcup_{k=1}^N \Sigma^{(k)}$ be as in~\eqref{Sigma_intro}, choose for each $k \in \{ 1, \dots, N\}$ a number $\beta^{(k)}$ as in Proposition~\ref{proposition_tubular_neighbourhood}, and set $\beta := \min \{ \beta^{(1)}, \dots, \beta^{(N)} \}$.  Denote for $k \in \{ 1, \dots, N \}$ the $\beta$-neighborhood of $\Sigma^{(k)}$ by $\Sigma_\beta^{(k)}$, i.e. $\Sigma_\beta^{(k)} = \iota^{(k)}(\Sigma^{(k)} \times (-\beta, \beta))$, where $\iota^{(k)}$ is the map from~\eqref{def_iota} for $\Sigma^{(k)}$, and define for $\alpha < 0$ the functions $V^{(k)} = \frac{1}{2 \beta} \alpha \chi_{\Sigma_\beta^{(k)}}$, where $\chi_{\Sigma_\beta^{(k)}}$ denotes the characteristic function for $\Sigma_\beta^{(k)}$. Moreover, introduce for $\varepsilon \in (0, \beta]$ the scaled potentials $V_\varepsilon$ as in~\eqref{def_V_epsilon}, and the associated Schr\"odinger operators $H_{A,0,V_\varepsilon}$ (for $Q=0$) via the associated quadratic forms in~\eqref{def_H_V_eps_form_intro}. Likewise, the Schr\"odinger operator with a $\delta$-potential of constant strength $\alpha$ is denoted by $H_{A,0,\alpha}$ and it is associated with the form in~\eqref{def_H_alpha_form_intro}.

First, a result from \cite{ELPO18} is used and conditions that ensure that the discrete spectrum of a magnetic Schr\"odinger operator with a homogeneous magnetic field $A(x,y) = \frac{1}{2}(-y,x)$ is non-empty are discussed, if the electric potential is supported in a small neighborhood of a wedge of opening angle $\phi \in (0, \pi)$, i.e.
\begin{equation*} 
  \Sigma = \{ (r \cos \theta, r \sin \theta): r > 0, \theta \in \{ 0, \phi\} \}.
\end{equation*}
Note that this $\Sigma$ falls into the class of networks considered in this paper, as it is the union of two rays $\Sigma^{(1)}$ and $\Sigma^{(2)}$ with joint endpoint at the origin. The associated potential $V_\varepsilon$ is given by
\begin{equation*}
  V_\varepsilon = \begin{cases} 2 \alpha \varepsilon^{-1} & \text{ in }\Omega_\varepsilon^{(1)} \cap \Omega_\varepsilon^{(2)} \\ \alpha \varepsilon^{-1}  & \text{ in }(\Omega_\varepsilon^{(1)} \cup \Omega_\varepsilon^{(2)})  \setminus (\Omega_\varepsilon^{(1)} \cap \Omega_\varepsilon^{(2)}), \\ 0 & \text{ else;} \end{cases}
\end{equation*}
note that there exists a constant $c = c(\phi)$ such that $\Omega_\varepsilon^{(1)} \cap \Omega_\varepsilon^{(2)} \subset B(0, c \varepsilon)$.

\begin{prop} \label{proposition_magnetic}
  Let $\alpha < 0$. In the situation described above, set 
  \begin{equation*}
    F_{\phi, \alpha}(x,y) := 1 + \frac{x^4}{4} - x^2 \Theta_{\delta, \alpha} + \frac{\alpha}{\sqrt{\pi}} x e^{-y^2\tan^2(\phi/2)}(1+\textup{erf}(y)),
  \end{equation*}
  where $\Theta_{\delta, \alpha} = \inf \sigma_\textup{ess}(H_{A, 0, \alpha}) > -\infty$ and $\textup{erf}$ is the error function. If $\varepsilon > 0$ is sufficiently small and $\inf_{x,y \in (0, \infty)} F_{\phi, \alpha}(x,y)<0$, then 
  \begin{equation*}
    \sigma_\textup{disc}(H_{A, 0, V_\varepsilon}) \cap (-\infty, \inf \sigma_\textup{ess}(H_{A, 0, V_\varepsilon})) \neq \emptyset.
  \end{equation*}
  In particular, this is the case, if one of the following conditions hold:
  \begin{itemize}
    \item[\textup{(a)}] $\phi \in (0, \frac{1}{3} \pi]$ and $|\alpha|$ is sufficiently small.
    \item[\textup{(b)}] $\phi \in (0, \frac{1}{8} \pi]$ and $|\alpha|$ is sufficiently large.
  \end{itemize}
\end{prop}
\begin{proof}
  First, taking the given assumptions into account, there exists a point $\lambda_0 \in \sigma_\textup{disc}(H_{A, 0, \alpha}) \cap (-\infty,\Theta_{\delta, \alpha})$, see \cite[Theorem~3 and Corollary~2]{ELPO18} (note that $\beta$ in \cite{ELPO18} is equal to $-\alpha$). Since by Lemma~\ref{lemma_convergence_spectrum} one has $\inf \sigma_\textup{ess}(H_{A, 0, V_\varepsilon}) \rightarrow \Theta_{\delta, \alpha}$, as $\varepsilon \searrow 0$, and there exists $\lambda_\varepsilon \in \sigma_{\textup{disc}}(H_{A,0,V_\varepsilon})$  such that $\lambda_\varepsilon \rightarrow \lambda_0$, as $\varepsilon \searrow 0$, one concludes that $\lambda_\varepsilon < \inf \sigma_\textup{ess}(H_{A, 0, V_\varepsilon})$.
\end{proof}

In the following proposition it is shown how, in a similar vein as in \cite[Corollary~1.3]{BFHSL24}, a geometric optimization result on $H_{0,0,\alpha}$ with a $\delta$-interaction supported on a star graph in $\mathbb{R}^2$ from \cite{EL18} can be transferred to $H_{0,0,V_\varepsilon}$. Note that in a similar way as in Proposition~\ref{proposition_star_graph} the optimization results for Laplacians with $\delta$-potentials from \cite{BLS23, E05, L19} can be transferred to Schr\"odinger operators with regular potentials. In particular, in these references also situations in other space dimensions and for more general real-valued $\alpha \in L^p(\Sigma) + L^\infty(\Sigma)$ with $p>1$ for $n=2$ and $p \geq n-1$ for $n \geq 3$ are contained. 

To formulate the mentioned optimization result for Schr\"odinger operators with potentials supported in a tubular neighborhood of a star graph, let $N \in \mathbb{N}$, $N>2$, be fixed and $\phi_1, \dots \phi_N \in (0, 2 \pi)$ such that $\phi_1 + \dots + \phi_N = 2\pi$, and set $\widetilde{\phi}_k := \sum_{j=1}^k \phi_j$, $k \in\{ 1, \dots, N\}$. Let $\Sigma \subset \mathbb{R}^2$ be a star graph, where each of the edges has length $L \in (0, \infty)$, i.e.
\begin{equation*}
  \Sigma = \{ (r \cos \theta, r \sin \theta): r \in (0, L), \theta \in \{ \widetilde{\phi}_1, \dots, \widetilde{\phi}_N \} \}.
\end{equation*}
Moreover, denote by $\Gamma$ the symmetric star graph with $N$ edges of length $L$, i.e. $\Gamma$ is as $\Sigma$ with $\phi_1 = \dots = \phi_N = \frac{2 \pi}{N}$. Note that both graphs $\Sigma$ and $\Gamma$ are networks as in~\eqref{Sigma_intro}, as they are unions of $N$ line segments $\Sigma^{(1)}, \dots, \Sigma^{(N)}$ and $\Gamma^{(1)}, \dots, \Gamma^{(N)}$, respectively. Let $V_\varepsilon$ be as described after Lemma~\ref{lemma_convergence_spectrum} and let $W_\varepsilon$ be constructed in the same way, but with the network $\Sigma$ being replaced by $\Gamma$ with the same $\alpha$ and the same $\beta$ as in the definition of $V_\varepsilon$, and denote the associated Schr\"odinger operator by  $H_{0,0,W_\varepsilon}$. Note that, as in the previous result, $\max |V_\varepsilon| \leq N |\alpha| \varepsilon^{-1}$ and, with a constant $c = c(\phi_1, \dots, \phi_N)$, the area in which $|V_\varepsilon| \geq |\alpha| \varepsilon^{-1}$, is contained in $B(0, c \varepsilon)$; a similar statement holds for $W_\varepsilon$. 
As $\Sigma$ and $\Gamma$ are compact, $V_\varepsilon$ and $W_\varepsilon$ are bounded and compactly supported and thus $\sigma_\textup{ess}(H_{0,0,V_\varepsilon})=\sigma_\textup{ess}(H_{0,0,W_\varepsilon})=[0, \infty)$. One has the following optimization result for the lowest eigenvalue of $H_{0,0,V_\varepsilon}$.

\begin{prop} \label{proposition_star_graph}
  Let $\alpha < 0$. In the situation described above, for all $\varepsilon>0$ sufficiently small $H_{0, 0, V_\varepsilon}$ and $H_{0,0,W_\varepsilon}$ have negative  eigenvalues $\lambda_\varepsilon(\Sigma) := \inf \sigma_{\textup{disc}}(H_{0,0, V_\varepsilon})$ and $\lambda_\varepsilon(\Gamma) := \inf \sigma_{\textup{disc}}(H_{0,0, W_\varepsilon})$,
  \begin{equation} \label{star_graph_eigenvalues}
    \lambda_\varepsilon(\Sigma) \leq \lambda_\varepsilon(\Gamma),
  \end{equation}
  and equality is achieved if and only if $\Sigma$ and $\Gamma$ are congruent to each other.
\end{prop}
\begin{proof}
  First, by \cite[Theorem~1.1]{EL18} the Schr\"odinger operators $H_{0,0,\alpha}(\Sigma)$ and $H_{0,0,\alpha}(\Gamma)$ with $\delta$-potentials supported on $\Sigma$ and $\Gamma$, respectively, have negative discrete eigenvalues and for their smallest eigenvalues $\lambda_0(\Sigma)$ and $\lambda_0(\Gamma)$ the inequality
  \begin{equation*}
    \lambda_0(\Sigma) \leq \lambda_0(\Gamma) 
  \end{equation*}
  holds, where equality is achieved if and only if $\Sigma$ and $\Gamma$ are congruent to each other. By Lemma~\ref{lemma_convergence_spectrum} one has for $\varepsilon>0$ sufficiently small that $H_{0, 0, V_\varepsilon}$ and $H_{0,0,W_\varepsilon}$ admit negative eigenvalues, that~\eqref{star_graph_eigenvalues} is true, and that equality holds if and only if $\Sigma$ and $\Gamma$ are congruent to each other.
\end{proof}

Finally, it is shown how corners in the network $\Sigma \subset \mathbb{R}^2$ can create geometrically induced eigenvalues also for operators with potentials that are supported in a tubular neighborhood of $\Sigma$. This will be done in the situation from \cite{FP20}, when $\Sigma$ is the boundary of a compact domain with a cusp; a similar result for piecewise $C^4$-smooth Lipschitz curves from \cite[Theorem~11]{BCP25} can be handled in the same way.
Assume that $\Sigma$ is a compact and closed curve that is $C^4$-smooth except at the origin and that there exists $\varepsilon_2 \in (0,1)$ and $d>1$ such that 
\begin{equation*}
  \Sigma \cap (-\varepsilon_2, \varepsilon_2)^2 = \{ (x,y): x \in [0, \varepsilon_2), |y|=x^d \}.
\end{equation*}
Note that this curve with a cusp at the origin can be realized by gluing together two admissible curves $\Sigma^{(1)}$ and $\Sigma^{(2)}$ at the origin and hence, it falls into the class of networks considered in the present paper. 
In this case, the associated potential $V_\varepsilon$ is given by
\begin{equation*}
  V_\varepsilon = \begin{cases} 2 \alpha \varepsilon^{-1} & \text{ in }\Omega_\varepsilon^{(1)} \cap \Omega_\varepsilon^{(2)} \\ \alpha \varepsilon^{-1}  & \text{ in }(\Omega_\varepsilon^{(1)} \cup \Omega_\varepsilon^{(2)})  \setminus (\Omega_\varepsilon^{(1)} \cap \Omega_\varepsilon^{(2)}), \\ 0 & \text{ else;} \end{cases}
\end{equation*}
note that there exists a constant $c = c(\varepsilon) \in (0,1)$, $c(\varepsilon) \rightarrow 1$ for $\varepsilon \searrow 0$, such that $\Omega_\varepsilon^{(1)} \cap \Omega_\varepsilon^{(2)}$ is contained in the $\varepsilon$-neighborhood of $|y|=x^d$ for $x \in (0, c \varepsilon^{1/d})$; this is a larger set than in the previous two examples.
To formulate the next result, one has to introduce in $L^2(0, \infty)$ the differential operator
\begin{equation} \label{def_B}
  B f (x) = -f''(x) + x^d f(x)
\end{equation}
defined on functions that satisfy the Dirichlet condition $f(0) = 0$. Then, the following result is a consequence of the analysis from \cite{FP20}.

\begin{prop} \label{proposition_cusp}
  Let $M \in \mathbb{N}$ be fixed. Then, in the situation described above, there exists $\alpha = \alpha(M)$ sufficiently negative and $\varepsilon_3 = \varepsilon_3(\alpha(M), M)$ such that for all $\varepsilon \in (0, \varepsilon_3)$ the operator $H_{0, 0, V_\varepsilon}$ has $M$ negative eigenvalues $\lambda_{k,\varepsilon}$, $k \in \{ 1, \dots, M\}$, that behave as
  \begin{equation} \label{cusp_eigenvalues}
    \lambda_{k,\varepsilon} \approx -\alpha^2 + 2^{2/(d+2)} E_k(B) \alpha^{6/(d+2)},
  \end{equation}
  where $E_k(B)$ is the $k$-th eigenvalue of the operator $B$ defined by~\eqref{def_B}.
\end{prop}
\begin{proof}
  First, by \cite[Theorem~1]{FP20} the Schr\"odinger operator $H_{0,0,\alpha}$  with a $\delta$-potential supported on $\Sigma$ has, for $\alpha < 0$ sufficiently negative, $M$ negative discrete eigenvalues $\lambda_{k,0}$, $k \in \{ 1, \dots, M \}$, that behave as
  \begin{equation*} 
    \lambda_{k,0} = -\alpha^2 + 2^{2/(d+2)} E_k(B) \alpha^{6/(d+2)} + \mathcal{O}(\alpha^{6/(d+2)-\eta}), 
  \end{equation*}
  where $\eta = \min\big\{ \frac{d-1}{2(d+2)}, \frac{2(d-1)}{(d+1)(d+2)} \big\}$.
  Thus, it follows from Lemma~\ref{lemma_convergence_spectrum} that for all sufficiently small $\varepsilon > 0$ the operator $H_{0, 0, V_\varepsilon}$ has $M$ negative eigenvalues $\lambda_{k,\varepsilon}$, $k \in \{ 1, \dots, M\}$, that behave as in~\eqref{cusp_eigenvalues}.
\end{proof}

\subsection*{Acknowledgement.}
The author thanks the anonymous reviewers for constructive feedback that helped to improve the paper as well as Jussi Behrndt and Petr Siegl for fruitful discussions.

\subsection*{Funding statement.}
This research was funded in part by the Austrian Science Fund (FWF) 10.55776/P 33568-N. For the purpose of open access, the author has applied a CC BY public copyright licence to any Author Accepted Manuscript version arising from this submission.

\subsection*{Data availability statement}
Data sharing not applicable to this article as no datasets were generated or analysed during the current study.

\subsection*{Competing Interests}

The author has no competing interests to declare that are relevant to the content of this article.



\end{document}